\documentclass[12pt]{amsart}
\usepackage{amssymb}
\usepackage{amsbsy}
\usepackage{amscd}

\makeatletter
\renewcommand{\thesubsection}{\thesection(\@roman\c@subsection)}
\makeatother
\usepackage{verbatim}
\usepackage{version}
\usepackage{color}
\newenvironment{NB}{
\color{red}{\bf NB}. \footnotesize
}{}
\excludeversion{NB}
\newenvironment{NB2}{
\color{blue}{\bf NB}. \footnotesize
}{}
\newtheorem{Theorem}[equation]{Theorem}

\newtheorem{Lemma}[equation]{Lemma}
\newtheorem{Proposition}[equation]{Proposition}

\theoremstyle{definition}

\theoremstyle{remark}
\newtheorem{Remark}[equation]{Remark}

\numberwithin{equation}{section}

\newcommand{\thmref}[1]{Theorem~\ref{#1}}
\newcommand{\secref}[1]{\S\ref{#1}}
\newcommand{\lemref}[1]{Lemma~\ref{#1}}
\newcommand{\propref}[1]{Proposition~\ref{#1}}

\newcommand{\subsecref}[1]{\S\ref{#1}}

\newcommand{\remref}[1]{Remark~\ref{#1}}

\newcommand{\lsp}[2]{{\mskip-.3mu}{}^{#1}\mskip-1mu{#2}}
\newcommand{\defeq}{\overset{\operatorname{\scriptstyle def.}}{=}}
\newcommand{\C}{{\mathbb C}}
\newcommand{\Z}{{\mathbb Z}}
\newcommand{\Q}{{\mathbb Q}}

\newcommand{\SU}{\operatorname{\rm SU}}
\newcommand{\GL}{\operatorname{GL}}

\newcommand{\gl}{\operatorname{\mathfrak{gl}}}

\newcommand{\g}{{\mathfrak g}}

\newcommand{\End}{\operatorname{End}}
\newcommand{\Hom}{\operatorname{Hom}}
\newcommand{\Ext}{\operatorname{Ext}}
\newcommand{\Ker}{\operatorname{Ker}}

\newcommand{\Ima}{\operatorname{Im}}

\newcommand{\tr}{\operatorname{tr}}

\newcommand{\id}{\operatorname{id}}
\newcommand{\ve}{\varepsilon}
\newcommand{\vin}[1]{\operatorname{i}(#1)} 
\newcommand{\vout}[1]{\operatorname{o}(#1)} 
\newcommand{\bM}{{\mathbf M}} 
\newcommand{\M}{{\mathfrak M}} 
\newcommand{\Mreg}{\M^{\operatorname{reg}}}
\newcommand{\La}{{\mathfrak L}} 
\newcommand{\bv}{{\mathbf v}} 
\newcommand{\bw}{{\mathbf w}} 
\newcommand{\bC}{{\mathbf C}} 
\newcommand{\bA}{{\mathbf A}} 
\newcommand{\bI}{{\mathbf I}} 

\newcommand{\topdeg}{{\operatorname{top}}} 

\newcommand{\HomE}{\operatorname{E}}
\newcommand{\HomL}{\operatorname{L}}

\newcommand{\shfO}{\mathcal O}

\renewcommand{\MR}[1]{}

\newcommand{\fT}{\mathfrak T}

\newcommand{\cC}{\mathcal C}

\usepackage[bookmarks=false]{hyperref}

\usepackage{url}

\begin{document}

\author{Hiraku Nakajima}
\title[Quiver varieties and tensor products, II]
{Quiver varieties and tensor products, II
}
\address{Research Institute for Mathematical Sciences,
Kyoto University, Kyoto 606-8502,
Japan}
\email{nakajima@kurims.kyoto-u.ac.jp}
\thanks{Supported by the Grant-in-aid
for Scientific Research (No.23340005), JSPS, Japan.
}
\begin{abstract}
  We define a family of homomorphisms on a collection of convolution
  algebras associated with quiver varieties, which gives a kind of
  coproduct on the Yangian associated with a symmetric Kac-Moody Lie
  algebra. We study its property using perverse sheaves.
\end{abstract}

\maketitle
\tableofcontents

\section*{Introduction}

In the conference the author explained his joint work with Guay on a
construction of a coproduct on the Yangian $Y(\g)$ associated with an
affine Kac-Moody Lie algebra $\g$.
It is a natural generalization of the coproduct on the usual Yangian
$Y(\g)$ for a finite dimensional complex simple Lie algebra $\g$ given
by Drinfeld \cite{Drinfeld}.
Its definition is motivated also by a recent work of Maulik
and Okounkov \cite{MO} on a geometric construction of a tensor product
structure on equivariant homology groups of holomorphic symplectic
varieties, in particular of quiver varieties.
The purpose of this paper is to explain this geometric background.

For quiver varieties of finite type, the geometric coproduct
corresponding to the Drinfeld coproduct on Yangian $Y(\g)$, or more
precisely the quantum affine algebra $U_q(\widehat\g)$, was studied in
\cite{VV-std,Na-Tensor,VV2}. (And one corresponding to the coproduct
on $\g$ was studied also in \cite{Malkin}.)
But the results depend on the algebraic definition of the coproduct.
As it is not known how to define a coproduct on $Y(\g)$ for an
arbitrary Kac-Moody Lie algebra $\g$, the results cannot be
generalized to other types.

In this paper, we take a geometric approach and define a kind of a
coproduct on convolution algebras associated with quiver varieties
together with a $\C^*$-action preserving the holomorphic symplectic
form, and study its properties using perverse sheaves.
%

In fact, we have an ambiguity in the definition of the coproduct,
and we have a family of coproducts $\Delta_c$, parametrized by $c$ in
a certain affine space.
This ambiguity of the coproduct was already noticed in \cite[Remark
in \S5.2]{VV2}.
Maulik-Okounkov theory gives a canonical choice of $c$ for a quiver
variety of an arbitrary type,
and gives the formula of $\Delta_c$ on standard generators of $Y(\g)$.
Therefore we can take the formula as a definition of the coproduct and
check its compatibility with the defining relations of $Y(\g)$.
This will be done for an affine Kac-Moody Lie algebra $\g$ as we
explained in the conference.
(The formula is a consequence of results in \cite{MO}, and hence is not
explained here.)

Although there is a natural choice, the author hopes that our
framework, considering also other possibilities for $\Delta$, is
suitable for a modification to other examples of convolution algebras
when geometry does not give us such a canonical choice. (For example,
the AGT conjecture for a general group. See \cite{NaAGT}.)

Remark also that our construction is specific for $Y(\g)$, and is not
clear how to apply for a quantum loop algebra $U_q(\mathbf L\g)$. We
need to replace cohomology groups by $K$ groups to deal with the
latter, but many of our arguments work only for cohomology groups.

Finally let us comment on a difference on the coproduct for quiver
varieties of finite type and other types. 
A coproduct on an algebra $A$ usually means an algebra homomorphism
$\Delta\colon A\to A\otimes A$ satisfying the coassociativity. In our
setting the algebra $A$ depends on the dimension vector, or
equivalently dominant weight $\bw$. Hence $\Delta$ is supposed to be a
homomorphism from the algebra $A(\bw)$ for $\bw$ to the tensor product
$A(\bw^1)\otimes A(\bw^2)$ with $\bw = \bw^1+\bw^2$. For a quiver of
type $ADE$, this is true, but not in general. See \remref{rem:diff}
for the crucial point. The target of $\Delta$ is, in general, larger
than $A(\bw^1)\otimes A(\bw^2)$.
Fortunately this difference is not essential, for example, study of
tensor product structures of representations of Yangians.

\subsection*{Notations}

The definition and notation of quiver varieties related to a coproduct
are as in \cite{Na-Tensor}, except the followings:
\begin{itemize}
\item Linear maps $i$, $j$ are denoted by $a$, $b$ here.
\item A quiver possibly contains edge loops. Roots are defined as in
  \cite[\S2]{CB}. They are obtained from coordinate vectors at loop
  free vertices or $\pm$ elements in the fundamental region by
  applying some sequences of reflections at loop free vertices.
\item Varieties $\mathfrak Z$, $\widetilde{\mathfrak Z}$ are denote by
  $\fT$, $\widetilde\fT$ here.
\end{itemize}

We say a quiver is of {\it finite type}, if its underlying graph is of
type $ADE$. We way it is of {\it affine type}, if it is Jordan quiver
or its underlying graph is an extended Dynkin diagram of type $ADE$.


For $\bv = (v_i)$, $\bv' = (v'_i)\in\Z^{I}$, we say $\bv\le\bv'$ if
$v_i\le v'_i$ for any $i\in I$.

For a variety $X$, $H_*(X)$ denote its Borel-Moore homology group. It
is the dual to $H^*_c(X)$ the cohomology group with compact support.

We will use the homology group $H_*(L)$ of a closed variety $L$ in a
smooth variety $M$ in several contexts. There is often a preferred degree
in the context, which is written as `$\topdeg$' below. For example, if
$L$ is lagrangian, it is $\dim_\C M$.
If $M$ has several components $M_\alpha$ of various dimensions, we
mean $H_{\topdeg}(L)$ to be the direct sum of $H_{\topdeg}(L\cap
M_\alpha)$, though the degree `$\topdeg$' changes for each $L\cap
M_\alpha$.

Let $D(X)$ denote the bounded derived category of complexes of
constructible $\C$-sheaves on $X$.
When $X$ is smooth, $\cC_X\in D(X)$ denote the constant sheaf on $X$
shifted by $\dim X$. If $X$ is a disjoint union of smooth varieties
$X_\alpha$ with various dimensions, we understand $\cC_X$ as the
direct sum of $\cC_{X_\alpha}$.

The intersection cohomology ($IC$ for short) complex associated with a
smooth locally closed subvariety $Y\subset X$ and a local system
$\rho$ on $Y$ is denoted by $IC(Y,\rho)$ or $IC(\overline{Y},\rho)$.
If $\rho$ is the trivial rank $1$ local system, we simply denote it
by $IC(Y)$ or $IC(\overline{Y})$.

\section{Quiver varieties}

In this section we fix the notation for quiver varieties. See
\cite{Na-quiver,Na-alg} for detail.

Suppose that a finite graph is given. Let $I$ be the set of vertices
and $E$ the set of edges.
In \cite{Na-quiver,Na-alg} the author assumed that the graph does not
contain edge loops (i.e., no edges joining a vertex with itself), but
most of results (in particular definitions, natural morphisms, etc)
hold without this assumption.

Let $H$ be the set of pairs consisting of an edge together with its
orientation. So we have $\# H = 2\# E$.
For $h\in H$, we denote by $\vin{h}$ (resp.\ $\vout{h}$) the incoming
(resp.\ outgoing) vertex of $h$.  For $h\in H$ we denote by $\overline
h$ the same edge as $h$ with the reverse orientation.
Choose and fix an orientation $\Omega$ of the graph,
i.e., a subset $\Omega\subset H$ such that
$\overline\Omega\cup\Omega = H$, $\Omega\cap\overline\Omega = \emptyset$.
The pair $(I,\Omega)$ is called a {\it quiver}.
\begin{NB}
Let us define matrices $\bA_\Omega$ and $\bA_{\overline\Omega}$ by
\begin{equation}\label{eq:Aomega}
\begin{split}
  (\bA_\Omega)_{ij} 
  & \defeq \# \{ h\in \Omega \mid \vin{h} = i, \vout{h} = j\}, \\
  (\bA_{\overline\Omega})_{ij} 
  & \defeq \# \{ h\in \overline{\Omega} \mid \vin{h} = i, \vout{h} = j\}.
\end{split}
\end{equation}
So we have $\bC = 2\bI - (\bA_\Omega + \bA_{\overline\Omega})$,
$\lsp{t}\bA_\Omega = \bA_{\overline\Omega}$.  
\end{NB}

Let $V = (V_i)_{i\in I}$ be a finite dimensional $I$-graded vector
space over $\C$. The dimension of $V$ is a vector
\[
  \dim V = (\dim V_i)_{i\in I}\in \mathbb Z_{\ge 0}^I.
\]

If $V^1$ and $V^2$ are $I$-graded vector spaces, we define vector spaces by
\begin{gather*}
  \HomL(V^1, V^2) \defeq
  \bigoplus_{i\in I} \Hom(V^1_i, V^2_i), \quad
  \HomE(V^1, V^2) \defeq
  \bigoplus_{h\in H} \Hom(V^1_{\vout{h}}, V^2_{\vin{h}})
  .
\end{gather*}

For $B = (B_h) \in \HomE(V^1, V^2)$ and 
$C = (C_h) \in \HomE(V^2, V^3)$, let us define a multiplication of $B$
and $C$ by
\[
  CB \defeq \left(\sum_{\vin{h} = i} C_h B_{\overline h}\right)_i \in
  \HomL(V^1, V^3).
\]
Multiplications $ba$, $Ba$ of $a\in \HomL(V^1,V^2)$, $b\in\HomL(V^2,
V^3)$, $B\in \HomE(V^2, V^3)$ are defined in the obvious manner. If
$a\in\HomL(V^1, V^1)$, its trace $\tr(a)$ is understood as $\sum_i
\tr(a_i)$.

For two $I$-graded vector spaces $V$, $W$ with $\bv = \dim V$, $\bw =
\dim W$, we consider the vector space given by
\begin{equation*}\label{def:bM}
  \bM \equiv \bM(\bv,\bw) \defeq
  \HomE(V,V)\oplus \HomL(W,V)\oplus \HomL(V,W),
\end{equation*}
where we use the notation $\bM$ when $\bv$, $\bw$ are clear in the
context.
%
The above three components for an element of $\bM$ will be denoted by
$B = \bigoplus B_h $, $a = \bigoplus a_i$, $b = \bigoplus b_i$
respectively.


The orientation $\Omega$ defines a function $\varepsilon\colon H \to
\{ \pm 1\}$ by $\varepsilon(h) = 1$ if $h\in\Omega$, $\varepsilon(h)=
-1$ if $h\in\overline\Omega$. We consider $\varepsilon$ as an element
of $\HomL(V, V)$.
Let us define a symplectic form $\omega$ on $\bM$ by
\begin{equation*}
        \omega((B, a, b), (B', a', b'))
        \defeq \tr(\varepsilon B\, B') + \tr(a b' - a' b).
\label{def:symplectic}\end{equation*}

Let $G \equiv G_\bv$ be an algebraic group defined by
\begin{equation*}
   G \equiv G_\bv \defeq \prod_i \GL(V_i).
\end{equation*}
Its Lie algebra is the direct sum $\bigoplus_i \gl(V_i)$.
The group $G$ acts on $\bM$ by
\begin{equation*}\label{eq:Kaction}
  (B,a,b) \mapsto g\cdot (B,a,b)
  \defeq (g B g^{-1}, ga, bg^{-1})
\end{equation*}
preserving the symplectic structure.

The moment map vanishing at the origin is given by
\begin{equation*}
  \mu(B, a, b) = \varepsilon B\, B + ab \in \HomL(V, V),
\end{equation*}
where the dual of the Lie algebra of $G$ is identified with
$\HomL(V, V)$ via the trace.


We would like to consider a `symplectic quotient' of $\mu^{-1}(0)$
divided by $G$. 
However we cannot expect the set-theoretical quotient
to have a good property. Therefore we consider the quotient using the
geometric invariant theory. Then the quotient depends on an additional
parameter $\zeta = (\zeta_i)_{i\in I}\in \Z^I$ as follows:
Let us define a character of $G$ by
\begin{equation*}
  \chi_\zeta(g) \defeq \prod_{i\in I} \left(\det g_i\right)^{-\zeta_i}.
\end{equation*}
Let $A(\mu^{-1}(0))$ be the coodinate ring of the affine variety
$\mu^{-1}(0)$. Set
\begin{equation*}
   A(\mu^{-1}(0))^{G,\chi_\zeta^n}
   \defeq \{ f\in A(\mu^{-1}(0)) \mid f(g\cdot(B,a,b))
   = \chi_\zeta(g)^n f((B,a,b)) \}.
\end{equation*}
The direct sum with respect to $n\in\Z_{\ge 0}$ is a graded algebra,
hence we can define 
\begin{equation*}
   \M_\zeta \equiv \M_\zeta(\bv,\bw) \equiv \M_\zeta(V,W)
   \defeq \operatorname{Proj}(\bigoplus_{n\ge 0}
   A(\mu^{-1}(0))^{G,\chi_\zeta^n}).
\end{equation*}
This is the {\it quiver variety\/} introduced in \cite{Na-quiver}.
Since this space is unchanged when we replace $\chi$ by a positive power
$\chi^N$ ($N > 0$), this space is well-defined for $\zeta\in \Q^I$.
We call $\zeta$ a {\it stability parameter}.

We use two special stability parameters in this paper. When $\zeta =
0$, the corresponding $\M_0$ is an affine algebraic variety whose
coordinate ring consists of the $G$-invariant functions on
$\mu^{-1}(0)$.

Another choice is $\zeta_i = 1$ for all $i$. In this case, we denote
the corresponding variety simply by $\M$. The corresponding stability
condition is that an $I$-graded subspace $V'$ of $V$ invariant under
$B$ and contained in $\Ker b$ is $0$ \cite[Lemma~3.8]{Na-alg}. The
stability and semistability are equivalent in this case, and the
action of $G$ on the set $\mu^{-1}(0)^{\mathrm{s}}$ of stable points
is free, and $\M$ is the quotient $\mu^{-1}(0)^{\mathrm{s}}/G$. In
particular $\M$ is nonsingular.


\section{Tensor product varieties}\label{sec:}

Let $W^2\subset W$ be an $I$-graded subspace and $W^1 = W/W^2$ be the
quotient. We fix an isomorphism $W \cong W^1\oplus W^2$. We define a one
parameter subgroup $\lambda\colon \C^*\to G_W$ by
\(
  \lambda(t) = \id_{W^1} \oplus\; t \id_{W^2}.
\)
Then $\C^*$ acts on $\M$, $\M_0$ through $\lambda$.

We fix $\bv$, $\bw$ and $\bw^1 = \dim W^1$, $\bw^2 = \dim W^2$
throughout this paper.
Since we use several quiver varieties with different dimension
vectors, let us use the notation $\M(\bv^1,\bw^1)$, etc for those,
while the notation $\M$ means the original $\M(\bv,\bw)$.

\subsection{Fixed points}\label{subsec:fixed}

We consider the fixed point loci $\M^{\C^*}$,
$\M_0^{\C^*}$. The former decomposes as
\begin{equation}\label{eq:decomp1}
  \M^{\C^*} = \bigsqcup_{\bv=\bv^1+\bv^2} \M(\bv^1,\bw^1)\times \M(\bv^2,\bw^2)
\end{equation}
(see \cite[Lemma~3.2]{Na-Tensor}). The isomorphism is given by
considering the direct sum of $[B^1,a^1,b^1]\in\M(\bv^1,\bw^1)$
and $[B^2,a^2,b^2]\in \M(\bv^2,\bw^2)$ as a point in $\M$.
Since quiver varieties $\M(\bv^1,\bw^1)$, $\M(\bv^2,\bw^2)$ are
connected, this is a decomposition of $\M^{\C^*}$ into connected
components.

Let us study the second fixed point locus $\M_0^{\C^*}$. We have a
morphism
\begin{equation*}
  \sigma\colon \bigsqcup_{\bv = \bv^1+\bv^2}
  \M_0(\bv^1,\bw^1)\times \M_0(\bv^2,\bw^2) \to \M_0^{\C^*}
\end{equation*}
given by the direct sum as above. This cannot be an isomorphism unless
$\bv=0$ as the inverse image of $0$ consists of several points
corresponding to various decomposition $\bv = \bv^1+\bv^2$. This is
compensated by considering the direct limit $\M_0(\bw) = \bigcup_{\bv}
\M_0(\bv,\bw)$ if the underlying graph is of type $ADE$. But this
trick does not solve the problem yet in general.
For example, if the quiver is the Jordan quiver, and $\bv^1 = \bw^1 =
\bv^2 = \bw^2 = 1$, we have $\M_0(\bv^1,\bw^1) = \M_0(\bv^2,\bw^2) =
\C^2$, while $\M_0^{\C^*} = S^2(\C^2)$. The morphism $\sigma$ is the
quotient map $\C^2\times \C^2 \to S^2(\C^2) = (\C^2\times\C^2)/S_2$.
Let us study $\sigma$ further.

Using the stratification \cite[Lemma~3.27]{Na-alg} we decompose
$\M_0 = \M_0(\bv,\bw)$ as
\begin{equation}\label{eq:decomp2}
  \M_0 = \bigsqcup_{\bv^0} \Mreg_0(\bv^0,\bw)\times \M_0(\bv-\bv^0,0),
\end{equation}
where $\Mreg_0(\bv^0,\bw)$ is the open subvariety of $\M_0(\bv^0,\bw)$
consisting of closed free orbits, and $\M_0(\bv-\bv^0,0)$ is the
quiver variety associated with $W=0$.
For quiver varieties of type $ADE$, the factor $\M_0(\bv-\bv^0,0)$ is
a single point $0$.
It is nontrivial in general.
For example, if the quiver is the Jordan quiver, we have
$\M_0(\bv-\bv^0,0) = S^n(\C^2)$ where $n = \bv-\bv^0$. Then
\begin{Lemma}\label{lem:strat}
  \textup{(1)} The above stratification induces a stratification
  \begin{equation*}
    \M_0^{\C^*}
    = \bigsqcup_{\substack{\bv^0, \lsp{1}\bv,\lsp{2}\bv \\
        \bv^0 = \lsp{1}\bv+\lsp{2}\bv}}
    \Mreg_0(\lsp{1}\bv,\bw^1)\times \Mreg_0(\lsp{2}\bv,\bw^2)
    \times \M_0(\bv-\bv^0,0).
  \end{equation*}

  \textup{(2)} $\sigma$ is a surjective finite morphism.
\end{Lemma}

Thus the factor with $W=0$ appears twice in $\M_0(\bv^1,\bw^1)\times
\M_0(\bv^2,\bw^2)$ while it appears only once in $\M_0^{\C^*}$.

\begin{proof}
  (1) We consider $\Mreg_0(\bv^0,\bw)$ as an open subvariety in
  $\M(\bv^0,\bw)$ and restrict the decomposition
  \eqref{eq:decomp1}. Then it is easy to check that $(x,y)\in
  \M(\lsp{1}\bv,\bw^1)\times\M(\lsp{2}\bv,\bw^2)$ is contained in
  $\Mreg_0(\bv^0,\bw)$ if and only if $x$, $y$ are in
  $\Mreg_0(\lsp{1}\bv,\bw^1)$, $\Mreg_0(\lsp{2}\bv,\bw^2)$
  respectively. Thus $\Mreg_0(\bv^0,\bw)^{\C^*} =
  \Mreg_0(\lsp{1}\bv,\bw^1)\times \Mreg_0(\lsp{2}\bv,\bw^2)$.
  Now the assertion is clear as $\C^*$ acts trivially on the factor
  $\M_0(\bv-\bv^0,0)$.

(2) The coordinate ring of $\M_0$ is generated by the following two
types of functions:
\begin{itemize}
\item $\tr(B_{h_N}B_{h_{N-1}}\cdots B_{h_1}\colon
V_{\vout{h_1}} \to V_{\vin{h_N}} = V_{\vout{h_1}})$,
where $h_1$, \dots, $h_N$ is a cycle in our graph.
\item $\chi(b_{\vin{h_N}} B_{h_N} B_{h_{N-1}} \cdots B_{h_1}
a_{\vout{h_1}})$,
where $h_1$, \dots, $h_N$ is a path in our graph, and $\chi$ is a
linear form on $\Hom(W_{\vout{h_1}}, W_{\vin{h_N}})$.
\end{itemize}
Then the generators for $\M_0^{\C^*}$ are the first type functions and
second type functions with $\chi =
(\chi_1,\chi_2)\in\Hom(W_{\vout{h_1}}^1, W^1_{\vin{h_N}})
\oplus\Hom(W_{\vout{h_1}}^2, W^2_{\vin{h_N}})$.

If we pull back these functions by $\sigma$, they become sums of the
same types of functions for $\M_0(\bv^1,\bw^1)$ and
$\M_0(\bv^2,\bw^2)$. From this observation, we can easily see that
$\sigma$ is a finite morphism.
From (1) it is clearly surjective.
\end{proof}

\begin{Remark}\label{rem:diff}
  Let $Z(\bv^a,\bw^a)$ be the fiber product
  $\M(\bv^a,\bw^a)\times_{\M_0(\bv^a,\bw^a)} \M(\bv^a,\bw^a)$ for
  $a=1,2$.  The fiber product $\M^{\C^*}\times_{\M_0^{\C^*}}\M^{\C^*}$
  is larger than
the union of the products $Z(\bv^1,\bw^1)\times Z(\bv^2,\bw^2)$
in general.
For example, consider the Jordan quiver variety with $\bv^1 = \bv^2 =
\bw^1 = \bw^2 = 1$. Then $\M(\bv^a,\bw^a)$ is $\C^2$. The product
$Z(\bv^1,\bw^1)\times Z(\bv^2,\bw^2)$ is consisting of points $(p_1,
q_1, p_2, q_2)$ with $p_1 = q_1$, $p_2 = q_2$. On the other hand,
$\M^{\C^*}\times_{\M_0^{\C^*}}\M^{\C^*}$ contains also points with
$p_1 = q_2$, $p_2 = q_1$.

On the other hand, if the quiver is of type $ADE$, we do not have
the factor $\M_0(\bv-\bv^0,0)$, and they are the same.
\end{Remark}

\subsection{Review of \cite{Na-Tensor}}\label{subsec:variety}

In this subsection we recall results in \cite[\S3]{Na-Tensor}, with
emphasis on subvarieties in the affine quotient $\M_0$.

We first define the following varieties which were implicitly
introduced in \cite[\S3]{Na-Tensor}:
\begin{equation*}
  \begin{split}
    \fT_0 &\defeq \{ x \in \M_0 \mid
    \text{$\lim_{t\to 0} \lambda(t) x$ exists}\},
    \\
    \widetilde\fT_0 &\defeq \{ x \in \M_0 \mid
    \lim_{t\to 0} \lambda(t) x = 0\}.
  \end{split}
\end{equation*}
By the proof of \cite[Lemma~3.6]{Na-Tensor} we have the following: $x
= [B,a,b]$ is in $\fT_0$ (resp.\ $\widetilde\fT_0$) if and only if
\begin{itemize}
\item $b_{\vin{h_N}} B_{h_N}B_{h_{N-1}}\cdots B_{h_1}a_{\vout{h_1}}$
  maps $W^2_{\vout{h_1}}$ to $W^2_{\vin{h_N}}$ (resp.\
  $W^2_{\vout{h_1}}$ to $0$ and the whole $W_{\vout{h_1}}$ to $W^2_{\vin{h_N}}$)
  for any path in the doubled quiver.
\end{itemize}
From this description it also follows that $\fT_0$,
$\widetilde\fT_0$ are closed subvarieties in $\M_0$.

We have the inclusion
\(
  i\colon \fT_0\to \M_0
\)
and the projection
\(
  p\colon \fT_0\to \M_0^{\C^*}
\)
defined by taking $\lim_{t\to 0}\lambda(t)x$.
The latter is defined as $\M_0$ is affine.

We define $\fT \defeq \pi^{-1}(\fT_0)$, $\widetilde\fT \defeq
\pi^{-1}(\widetilde\fT_0)$. These definitions coincide with ones in
\cite[\S3]{Na-Tensor}.
Note that we do not have an analog of $p\colon\fT_0\to
\M_0^{\C^*}$ for $\fT$.
Instead we have a decomposition
\begin{equation}\label{eq:decomp}
  \fT 
  = \bigsqcup_{\bv=\bv^1+\bv^2} \fT(\bv^1,\bw^1;\bv^2,\bw^2)
\end{equation}
into locally closed subvarieties,
and the projection 
\begin{equation}\label{eq:vector}
  p_{(\bv^1,\bv^2)}\colon \fT(\bv^1,\bw^1;\bv^2,\bw^2)\to
  \M(\bv^1,\bw^1)\times\M(\bv^2,\bw^2),
\end{equation}
which is a vector bundle. These are defined by considering the limit
$\lim_{t\to 0}\lambda(t)x$. Note that they intersect in their
closures, contrary to \eqref{eq:decomp1}, which was the decomposition
into connected components.
Since pieces in \eqref{eq:decomp} are mapped to different components,
$p_{(\bv^1,\bv^2)}$'s do not give a morphism defined on $\fT$.

As a vector bundle, $\fT(\bv^1,\bw^1;\bv^2,\bw^2)$ is the subbundle of
the normal bundle of $\M(\bv^1,\bw^1)\times\M(\bv^2,\bw^2)$ in
$\M$ consisting of positive weight spaces.
Its rank is half of the codimension of
$\M(\bv^1,\bw^1)\times\M(\bv^2,\bw^2)$.
In fact, the restriction of the tangent space of $\M$ to
$\M(\bv^1,\bw^1)\times\M(\bv^2,\bw^2)$ decomposes into weight $\pm 1$
and $0$ spaces such that
\begin{itemize}
\item the weight $0$ subspace gives the tangent bundle
of $\M(\bv^1,\bw^1)\times\M(\bv^2,\bw^2)$,
\item the weight $1$ and $-1$ subspaces are dual to each other with
  respect to the symplectic form on $\M$.
\end{itemize}

We define a partial order $<$ on the set $\{ (\bv^1,\bv^2) \mid
\bv^1+\bv^2=\bv\}$ defined by $(\bv^1,\bv^2) \le (\bv^{\prime
  1},\bv^{\prime 2})$ if and only if $\bv^1\le \bv^{\prime 1}$.
We extend it to a total order and denote it also by $<$. Let
\begin{equation*}
  \fT_{\le (\bv^1,\bv^2)} \defeq \bigcup_{(\bv^{\prime 1},\bv^{\prime 2})\le
  (\bv^1,\bv^2)} \fT(\bv^1,\bw^1;\bv^2,\bw^2),
\end{equation*}
and let $\fT_{<(\bv^1,\bv^2)}$ be the union obtained similarly by
replacing $\le$ by $<$.
Then $\fT_{\le(\bv^1,\bv^2)}$, $\fT_{<(\bv^1,\bv^2)}$ are closed
subvarieties in $\fT$.

\subsection{The fiber product $Z_{\fT}$}

We introduce one more variety, following \cite{MO}. Let us consider
$\fT_0$ as a subvariety in $\M_0\times \M_0^{\C^*}$, where the
projection to the second factor is given by $p$. Let
$Z_{\fT}\subset\M\times \M^{\C^*}$ be the inverse image of $\fT_0$
under the restriction of the projective morphism $\pi\times\pi$. This
variety is an analog of the variety $Z = \M\times_{\M_0}\M$ introduced
in \cite[\S7]{Na-alg}.
Note that $Z_{\fT}$ is also given as a fiber product
$\fT\times_{\M_0^{\C^*}}\M^{\C^*}$, where $\fT\to \M_0^{\C^*}$ is given
by the composition of $\pi\colon \fT\to \fT_0$ and
$p\colon\fT_0\to\M_0^{\C^*}$.
We will consider a cycle in $Z_{\fT}$ as a correspondence in
$\M\times\M^{\C^*}$ later.
Note that the projection $p_1\colon Z_{\fT}\to \M$ is proper, but
$p_2\colon Z_{\fT}\to \M^{\C^*}$ is not.

We consider the two decompositions (\ref{eq:decomp1}, \ref{eq:decomp}).
For brevity, we change the notation as
\begin{equation*}
  \M^{\C^*} = \bigsqcup_\alpha \M_\alpha, \quad
  \fT = \bigsqcup_\alpha \fT_\alpha.
\end{equation*}
We also recall
\begin{equation*}
  \fT_{\le\alpha} = \bigsqcup_{\beta\le \alpha} \fT_\beta, \qquad
  \fT_{<\alpha} = \bigsqcup_{\beta< \alpha} \fT_\beta
\end{equation*}
These are closed subvarieties in $\fT$.

Then they induce a decomposition
\begin{equation*}
  Z_{\fT} = \bigsqcup_{\alpha,\beta}
  Z_{\fT,\alpha,\beta},
\end{equation*}
with
\begin{equation*}
  Z_{\fT,\alpha,\beta}
  \defeq Z_{\fT}\cap \left(\fT_\alpha\times \M_\beta\right).
\end{equation*}
We have the corresponding decomposition
\begin{equation*}
  Z^{\C^*} = \M^{\C^*}\times_{\M_0^{\C^*}}\M^{\C^*}
  = \bigsqcup_{\alpha,\beta} Z_{\alpha,\beta}
\end{equation*}
induced from the the decomposition of the first and second factors.

We also define
\begin{equation*}
  Z_{\fT,\le\alpha,\beta} \defeq Z_{\fT}\cap
  \left(\fT_{\le\alpha}\times \M_\beta\right), \qquad
  Z_{\fT,<\alpha,\beta} \defeq Z_{\fT}\cap
  \left(\fT_{<\alpha}\times \M_\beta\right).
\end{equation*}
They are closed subvarieties in $Z_{\fT}$ and $Z_{\fT,\alpha,\beta}$ is
an open subvariety in $Z_{\fT,\le\alpha,\beta}$.
On the other hand, each $Z_{\alpha,\beta}$ is a closed subvariety
in $\M_\alpha\times\M_\beta$.

Each piece $Z_{\fT,\alpha,\beta}$ is a vector bundle over
\(
  Z_{\alpha,\beta},
\)
which is pull-back of $\fT_\alpha\to \M_\alpha$. Therefore its rank is
half of the codimension of $\M_\alpha$ in $\M$.

\begin{Proposition}\label{prop:irr}
  \textup{(1)} Each irreducible component of $Z_{\fT,\alpha,\beta}$ is
  at most half dimensional in $\M\times\M^{\C^*}$, and hence the same
  is true for $Z_{\fT}$.

  \textup{(2)} Irreducible components of $Z_{\fT}$ of half dimension
  are lagrangian subvarieties in $\M\times\M^{\C^*}$.
\end{Proposition}

Here $\M^{\C^*}$ has several connected components of various
dimensions, so the above more precisely meant half dimensional in each
component $\M\times\M_\beta$.

\begin{proof}
  (1) It is known that $\pi\colon \M\to\M_0$ is semismall, if we
  replace the target by the image $\pi(\M_0)$. (This is a consequence
  of \cite[6.11]{Na-quiver} as explained in
  \cite[2.23]{Na-branching}.)
  Therefore irreducible components of $Z = \M\times_{\M_0}\M$ are at
  most half dimensional in $\M\times\M$. As $\sigma$ is a finite
  morphism, the same is true for $Z^{\C^*}$.
  Now the assertion for $Z_{\fT,\alpha,\beta}$ follows as it is a
  vector bundle over $Z_{\alpha,\beta}$ whose rank is equal to the
  half of codimension of $\M_\alpha$.

  (2) This follows from the local description of $\pi$ in
  \cite[Theorem 3.3.2]{Na-qaff} which respects the symplectic form
  from its proof, together with the fact that $\pi^{-1}(0)$ is
  isotropic by the proof of \cite[Theorem 5.8]{Na-quiver}.
\end{proof}

\section{Coproduct}

In this section we define a kind of a coproduct on the convolution
algebra $H_*(Z)$. The target of $\Delta$ is, in general, larger than
the tensor product $H_*(Z(\bw^1))\otimes H_*(Z(\bw^2))$ as we
mentioned in the introduction.

\subsection{Convolution algebras}

Recall the fiber product $Z = \M\times_{\M_0}\M$. The convolution
product defines an algebra structure on $H_{*}(Z)$:
\begin{equation*}
  a\ast b \defeq p_{13*}(p_{12}^*(a)\cap p_{23}^*(b)), \quad
  a,b\in H_*(Z),
\end{equation*}
where $p_{ij}$ is the projection from $\M\times\M\times\M$ to the
product of the $i^{\mathrm{th}}$ and $j^{\mathrm{th}}$-factors, and
$Z$ is viewed as a subvariety in $\M\times\M$ for the cap
product. (See \cite[\S2.7]{CG} for more detail.)

As $\pi\colon\M\to\pi(\M)$ is a semismall morphism, the top degree
component $H_{\topdeg}(Z)$ is a subalgebra, where `$\topdeg$' is equal
to the complex dimension of $\M\times\M$. Moreover $H_*(Z)$ is a
graded algebra, where the degree $p$ elements are in $H_{\topdeg-p}(Z)$.
(See \cite[\S8.9]{CG}.)

Take $x\in \M_0$. We consider the inverse image $\pi^{-1}(x)\subset\M$
and denote it by $\M_x$. (When $x = 0$, this is denoted by $\La$
usually.) Then the convolution gives $\bigoplus H_{\topdeg-p}(\M_x)$ a
structure of a module of $H_*(Z)$. Here `$\topdeg$' is the difference
of complex dimensions of $\M$ and the stratum containing $x$.
\begin{NB}
  We can put any degree to $0$ as $H_{\topdeg-p}(Z) \otimes H_k(\M_x)\to
  H_{k-p}(\M_x)$.
\end{NB}

Similarly we can define a graded algebra structure on $H_*(Z^{\C^*}) =
\bigoplus H_{\topdeg-p}(Z^{\C^*})$, where `$\topdeg$' means the
complex dimension of $\M^{\C^*}\times\M^{\C^*}$, possibly different on
various connected components.
By \subsecref{subsec:fixed} it is close to
\begin{equation*}
  \bigoplus_{\bv^1+\bv^2 = \bv^2} H_*(Z(\bv^1,\bw^1))
  \otimes H_*(Z(\bv^2,\bw^2)),
\end{equation*}
but is different in general, as explained in \remref{rem:diff}.

We denote by $\M^{\C^*}_x$ the inverse image
$(\pi^{\C^*})^{-1}(x)$ in $\M^{\C^*}$ for $x\in\M_0^{\C^*}$.
Its homology $\bigoplus H_{\topdeg-p}(\M^{\C^*}_x)$ is a graded module
of $H_*(Z^{\C^*})$.
Here `$\topdeg$' is the difference of complex dimensions of
$\M^{\C^*}$ (resp.\ $\M$) and the stratum containing $x$.

\subsection{Convolution by $Z_{\fT}$}

Take $x\in \M_0^{\C^*}$. We consider the inverse image
$(p\circ\pi_{\fT})^{-1}(x)\subset\fT\subset\M$ and denote it by
$\fT_x$. (When $x = 0$, this is denoted by $\widetilde\fT$ in
\subsecref{subsec:variety}.)
By the convolution product its homology
$H_*(\fT_x) = \bigoplus H_{\topdeg-p}(\fT_x)$ is a graded module of
$H_*(Z)$.
 Here
`$\topdeg$' is the difference of complex dimensions of $\M$ and the
stratum containing $x$.
\begin{NB}
  This is natural as it is of expected dimension of $\fT_x$:
  $\fT_{\alpha,x}$ is a vector bundle over $\M_{\alpha,x}$ whose rank
  is half of the codimension of $\M^{\C^*}$ in $\M$.
\end{NB}%

Let $\fT_{\alpha,x}$, $\fT_{\le\alpha,x}$, $\fT_{<\alpha,x}$ be the
intersection of $\fT_x$ with $\fT_\alpha$, $\fT_{\le\alpha}$,
$\fT_{<\alpha}$ respectively. We have a short exact sequence
\begin{equation*}
  0 \to H_{\topdeg - p}(\fT_{<\alpha,x}) \to
  H_{\topdeg - p}(\fT_{\le\alpha,x}) \to H_{\topdeg - p}(\fT_{\alpha,x})\to 0.
\end{equation*}
(See \cite[\S3]{Na-Tensor} and \eqref{eq:short} below.) 

Let us restrict the projection $p_\alpha\colon
\fT_\alpha\to \M_\alpha$ in \eqref{eq:vector} to $\fT_{\alpha,x}$.
As $\pi^{\C^*}\circ p_\alpha = p\circ \pi$, it identifies
$\fT_{\alpha,x}$ with its inverse image of $\M_{\alpha,x} \defeq
\M_\alpha\cap \M^{\C^*}_x$.
Therefore we can replace the last term of the short exact sequence
by $H_{\topdeg-p}(\M_{\alpha,x})$ thanks to the Thom isomorphism:
\begin{equation}\label{eq:short1}
  0 \to H_{\topdeg - p}(\fT_{<\alpha,x}) \to
  H_{\topdeg - p}(\fT_{\le\alpha,x}) \to 
  H_{\topdeg-p}(\M_{\alpha,x})\to 0.
\end{equation}
Our convention of `$\topdeg$' is compatible for $\fT_x$ and $\M_x$ as
the rank of the vector bundle is the half of codimension of
$\M_\alpha$ in $\M$.
Since $\fT_{\le\alpha} = \fT$ when $\alpha$ is the maximal element, we get
\begin{Lemma}
  $H_{\topdeg-p}(\fT_x)$ has a filtration whose associated graded is
  isomorphic to $H_{\topdeg-p}(\M^{\C^*}_x)$.
\end{Lemma}

Choice of splittings $H_{\topdeg-p}(\fT_{\le\alpha,x})\leftarrow
H_{\topdeg-p}(\M_{\alpha,x})$ in \eqref{eq:short1} for all $\alpha$
gives an isomorphism $H_{\topdeg-p}(\fT_x)\cong
H_{\topdeg-p}(\M^{\C^*}_x)$. Our next goal is to understand the space
of all splittings in a geometric way.

For this purpose we consider the top degree homology group
$H_{\topdeg}(Z_{\fT})$. They are spanned by lagrangian irreducible
components of $Z_{\fT}$ by \propref{prop:irr}.

Let $c\in H_{\topdeg}(Z_\fT)$ and $p\in \Z$.
The convolution product
\begin{equation*}
  a \mapsto c\ast a \defeq p_{1*}(c\cap p_2^*(a))
\end{equation*}
defines an operator 
\begin{equation}\label{eq:op}
  c\ast\colon H_{\topdeg-p}(\M^{\C^*}_x) \to H_{\topdeg-p}(\fT_x),
\end{equation}
where $p_1$, $p_2$ are projections from $\M\times\M^{\C^*}$ to the
first and second factors.
The degree shift $p$ is preserved by the argument in
\cite[\S8.9]{CG}. (If we choose $c$ from $H_{\topdeg-k}(Z_\fT)$, the
convolution maps $H_{\topdeg-p}$ to $H_{\topdeg-p-k}$.)
Note also that the above operation is well-defined as $p_1$ is proper,
while the operator $p_{2*}(c\cap p_1^*(-))$ is not in this setting.

An arbitrary class $c\in H_{\topdeg}(Z_{\fT})$ does not give a
splitting of \eqref{eq:short1}, as it is nothing to do with the
decomposition $\fT = \bigsqcup\fT_{\alpha}$. Let us write down a
sufficient condition to give a splitting.

Since $c$ is in the top degree, it is a linear combination of
fundamental classes of lagrangian irreducible components of
$Z_{\fT}$. From \propref{prop:irr}(1), half-dimensional irreducible
components are closures of half-dimensional irreducible components of
$Z_{\fT,\beta,\alpha}$ for some pair $\alpha$, $\beta$.
Therefore we can write
\begin{equation*}
  c = \sum_{\alpha,\beta} c_{\beta,\alpha}.
\end{equation*}
Moreover its proof there, the latter are pull-backs of
half-dimensional irreducible components of $Z_{\beta,\alpha}$ under
the projection $p_\beta\times\id_{\M_\alpha}$.

We impose the following conditions on $c$:
\begin{subequations}\label{eq:cond}
  \begin{align}
    & \text{$c_{\beta,\alpha} = 0$ unless $\alpha\ge\beta$},
\\
      & c_{\alpha,\alpha} = [\overline{(p_\alpha\times
        \id_{\M_\alpha})^{-1}(\Delta_{\M_\alpha})}].
  \end{align}
\end{subequations}
The first condition also means that $c$ is is in the image of
\(
  \bigoplus_\alpha H_{\topdeg}(Z_{\fT,\le\alpha,\alpha})\to
  H_{\topdeg}(Z_{\fT}).
\)
Note that $\bigsqcup_\alpha Z_{\fT,\le\alpha,\alpha}$ is a disjoint
union of closed subvarieties in $Z_{\fT}$, and hence the push-forward 
homomorphism is defined.

\begin{Proposition}\label{prop:split}
  Let $c\in H_{\topdeg}(Z_{\fT})$ with the conditions
  \textup({\rm\ref{eq:cond} a,b}\textup). Then $c\ast$ is an
  isomorphism and gives a splitting of \eqref{eq:short1} for all
  $\alpha$.
\end{Proposition}

We will show the converse in \subsecref{subsec:desc}:
$c\ast$ gives a splitting if and only if $c$ satisfies
\eqref{eq:cond}.

\begin{proof}
By the first condition the operator $c\ast$ restricts to
$H_{\topdeg-p}(\M_{\alpha,x})\to
H_{\topdeg-p}(\fT_{\le\alpha,x})$. And by the second condition it
gives the identity if we compose $H_{\topdeg-p}(\fT_{\le\alpha,x})\to
H_{\topdeg-p}(\M_{\alpha,x})$.
Thus $c\ast$ gives a splitting of \eqref{eq:short1}.
\end{proof}

Next we construct the inverse of $c\ast$ also by a convolution
product. We consider
\begin{equation*}
  \fT_0^- \defeq \{ x\in\M_0\mid \text{$\lim_{t\to\infty}\lambda(t)x$ exists}
  \}, 
\end{equation*}
and the similarly defined variety $\fT^-$ also by replacing
$t\to 0$ by $t\to \infty$. We have the inclusion
\(
   i^-\colon \fT^-_0\to \M_0
\)
and the projection
\(
  p^-\colon \fT^-_0\to \M_0^{\C^*}.
\)
Note also that $\fT_0\cap\fT^-_0 = \M_0^{\C^*}$.

Let us define $Z_{\fT^-}$ as the fiber product
$\M^{\C^*}\times_{\M_0^{\C^*}}\fT^-$, and consider it as a subvariety
in $\M^{\C^*}\times\M$. We swap the first and second factors from
$Z_{\fT}$ as it becomes more natural when we consider a composite of
correspondences.

Since $p_1$ is proper on $Z_{\fT^-}\cap p_2^{-1}(\fT_{x})$, a class
$c^-\in H_{\topdeg}(Z_{\fT^-})$ defines the well-defined convolution
product $c^-\ast a = p_{1*}(c^-\cap p_2^*(a))$ for $a\in
H_{\topdeg-p}(\fT_x)$, and defines an operator
\begin{equation}\label{eq:c-}
  c^-\ast \colon H_{\topdeg-p}(\fT_x)\to H_{\topdeg-p}(\M^{\C^*}_x).
\end{equation}
By the associativity of the convolution product, the composite
$c^-\ast (c\ast \bullet) \in \End(H_{\topdeg-p}(\M^{\C^*}_x))$ is
given by the convolution of
\begin{equation*}
  c^-\ast c = p_{13*}(p_{12}^*(c^-)\cap p_{23}^*(c)),
\end{equation*}
where $p_{ij}$ is the projection from
$\M^{\C^*}\times\M\times\M^{\C^*}$ to the product of the
$i^{\mathrm{th}}$ and $j^{\mathrm{th}}$-factors.

\begin{Proposition}\label{prop:inverse}
  Suppose that $c\in H_{\topdeg}(Z_{\fT})$ satisfies the conditions
  \textup({\rm\ref{eq:cond} a,b}\textup). Then there exists a class
  $c^{-1}\in H_{\topdeg}(Z_{\fT^-})$ such that $c^{-1}\ast c$ is equal
  to $[\Delta_{\M^{\C^*}}]$.
\end{Proposition}

\begin{proof}
We have decomposition
\(
  \fT^- = \bigsqcup_\alpha \fT^-_\alpha,
\)
and the projection
\(
  p^-_\alpha\colon \fT^-_\alpha \to \M_\alpha.
\)
The index set $\{\alpha\}$ is the same as before, as it parametrizes
the connected components of $\M^{\C^*}$.

Since the order $<$ plays the opposite role for $\fT^-$,
\begin{equation*}
  \fT^-_{\ge\alpha} \defeq \bigsqcup_{\beta\ge\alpha} \fT^-_\beta,
\qquad
  \fT^-_{>\alpha} \defeq \bigsqcup_{\beta>\alpha} \fT^-_\beta,
\end{equation*}
are closed subvarieties in $\fT^-$.

We define $Z_{\fT^-,\gamma,\beta} \defeq Z_{\fT^-}\cap
(\M_\gamma\cap\fT^-_\beta)$ and $Z_{\fT^-,\gamma,\ge \beta}$ as
above. We then impose the following conditions on $c^- = \sum
c^-_{\gamma,\beta}$:
\begin{subequations}
  \begin{align*}
    & \text{$c^-_{\gamma,\beta} = 0$ unless $\gamma\le\beta$},
    \\
    & c^-_{\gamma,\gamma} = [\overline{(\id_{\M_\gamma}\times
      p_\gamma^-)^{-1}(\Delta_{\M_\gamma})}].
  \end{align*}
\end{subequations}
These conditions imply that
\(
  c^- \ast c 
\)
is unipotent, more precisely is upper triangular with respect to the
block decomposition $H_{\topdeg}(Z^{\C^*}) = \bigoplus_{\gamma,\alpha}
H_{\topdeg}(Z_{\gamma,\alpha})$, and the
$H_{\topdeg}(Z_{\alpha,\alpha})$ component is $[\Delta_{\M_\alpha}]$
for all $\alpha$. Noticing that we can represent $(c^-\ast c)^{-1}$ as
a class in $H_{\topdeg}(Z^{\C^*})$ by the convolution product, we
define $c^{-1} = (c^-\ast c)^{-1} \ast c^-\in H_{\topdeg}(Z_{\fT^-})$
to get $c^{-1}\ast c = [\Delta_{Z^{\C^*}}]$.

\begin{NB}
Earlier version: I think that this does not make sense.

Since $c^{-1}$ is the inverse matrix of $c$ if we represent $c$,
$c^{-1}$ in bases given by irreducible components of $Z_{\fT}$,
$Z_{\fT^-}$, the uniqueness is clear.
\end{NB}%
\end{proof}

\begin{Remark}\label{rem:inverse}
  If we consider the convolution product in the opposite order, we get
  \begin{equation*}
    c\ast c^{-1} \in H_{\topdeg}(\fT\times_{\M_0^{\C^*}}\fT^-),
  \end{equation*}
  where $\fT\to \M_0^{\C^*}$ (resp.\ $\fT\to \M_0^{\C^*}$) is $p\circ
  \pi$ (resp.\ $p^-\circ \pi$). In general, there are no inclusion
  relations between $\fT\times_{\M_0^{\C^*}}\fT^-$ and $Z =
  \M\times_{\M_0}\M$. Therefore the equality $c\ast c^{-1} =
  [\Delta_{\M}]$ does not make sense at the first sight.
  However the actual thing we need is the operator $c^{-}\ast$ in
  \eqref{eq:c-}.
  \propref{prop:inverse} implies that the composite $c^{-1}\ast(c\ast)$
  of the operator is the identity on $H_{\topdeg-p}(\M^{\C^*}_x)$ for
  each $x$.
  Then we have $c\ast (c^{-1}\ast)$ is also the identity on
  $H_{\topdeg-p}(\fT_x)$, as both $H_{\topdeg-p}(\M^{\C^*}_x)$ and
  $H_{\topdeg-p}(\fT_x)$ are vector spaces of same dimension.

  Later we will see that we do not loose any information when we
  consider $c^{-1}$ as such an operator. In particular, we will see
  that $c^{-1}$ is uniquely determined by $c$, i.e., we will prove the
  uniqueness of the left inverse in the proof of
  \thmref{thm:coproduct}.
\end{Remark}

\subsection{Coproduct by convolution}

We define a coproduct using the convolution in this subsection.

Let $c\in H_{\topdeg}(Z_{\fT})$ be a class satisfying the conditions
\eqref{eq:cond}. We take the class $c^{-1}\in H_{\topdeg}(Z_{\fT^-})$ as
in \propref{prop:inverse}. We define a homomorphism
\(
  \Delta_c \colon 
  H_{*}(Z)
  \to
  H_{*}(Z^{\C^*})
\)
by
\begin{equation}\label{eq:Delta}
  \Delta_c (\bullet) = c^{-1}\ast \bullet \ast c = 
  p_{14*}(p_{12}^*(c^{-1})\cap p_{23}^*(\bullet)\cap p_{34}^*(c)),
\end{equation}
where we consider the convolution product in $\M^{\C^*}\times \M\times
\M \times \M^{\C^*}$. This preserves the grading.

Since $c^{-1}\ast c = 1$, we have $\Delta_c(1) = 1$. But it is not
clear at this moment that $\Delta_c$ is an algebra homomorphism since we
do not know $c\ast c^{-1} = 1$, as we mentioned in
\remref{rem:inverse}. The proof is postponed until the next subsection.

\subsection{Sheaf-theoretic analysis}

In this subsection, we reformulate the result in the previous
subsection using perverse sheaves.

By \cite[\S8.9]{CG} we have a natural graded algebra isomorphism
\begin{equation*}
  H_*(Z) \cong \Ext^\bullet_{D(\M_0)}(\pi_!\cC_\M,\pi_!\cC_\M),
\end{equation*}
where the multiplication on the right hand side is given by the Yoneda
product and the grading is the natural one. Here the semismallness of
$\pi$ guarantees that the grading is preserved.

We have similarly
\begin{equation*}
  H_*(Z^{\C^*}) \cong 
  \Ext^\bullet_{D(\M_0^{\C^*})}
  (\pi^{\C^*}_!\cC_{\M^{\C^*}},\pi^{\C^*}_!\cC_{\M^{\C^*}}).
\end{equation*}
In this subsection we define a functor sending $\pi_!\C_{\M}$ to
$\pi^{\C^*}_!\C_{\M^{\C^*}}$ to give a homomorphism $H_*(Z)\to
H_*(Z^{\C^*})$ which coincides with $\Delta_c$.

For a later purpose, we slightly generalize the setting from the
previous subsection.
If $\bv'\le\bv$, we have a closed embedding $\M_0(\bv',\bw)\subset\M_0
= \M_0(\bv,\bw)$, given by adding the trivial representation with
dimension $\bv - \bv'$.
\begin{NB}
Then we can modify fiber products studied in the previous subsection
by replacing the second factor by $\M(\bv',\bw)$. For example,
a variant of $Z_{\fT}$ is
\( 
   \fT(\bv,\bw)\times_{\M_0(\bv,\bw)^{\C^*}} \M(\bv',\bw)^{\C^*}.
\)
\end{NB}

We consider the push-forward $\pi_!\cC_{\M(\bv',\bw)}$ as a complex in
$D(\M_0)$.
By the decomposition theorem \cite{BBD} it is a semisimple complex.
Furthermore $\pi_!\cC_{\M(\bv',\bw)}$ is a perverse sheaf, as
$\pi\colon\M(\bv',\bw)\to \pi(\M(\bv',\bw))$ is semismall \cite{BM}.
Let $P(\M_0)$ denote the full subcategory of $D(\M_0)$ consisting of
all perverse sheaves that are finite direct sums of perverse sheaves
$L$, which are isomorphic to direct summand of
$\pi_!\cC_{\M(\bv',\bw)}$ with various $\bv'$.

Replacing $\M_0$, $\M(\bv',\bw)$ by $\M_0^{\C^*}$,
$\M(\bv',\bw)^{\C^*}$ respectively, we introduce the full subcategory
$P(\M_0^{\C^*})$ of $D(\M_0^{\C^*})$ as above.
Here we replace $\pi$ by $\pi^{\C^*}\colon
\M(\bv',\bw)^{\C^*}\to\M_0(\bv',\bw)^{\C^*}$, which is the restriction
of $\pi$.

Let $i\colon \fT_0\to\M_0$ and $p\colon \fT_0\to\M_0^{\C^*}$ as in
\subsecref{subsec:variety}. We consider $p_!i^*\colon D(\M_0)\to
D(\M_0^{\C^*})$.
This is an analog of the restriction functor in \cite[\S4]{Lu-can2},
\cite[\S9.2]{Lu-book}, and was introduced in the quiver variety
setting in \cite[\S5]{VV2}.
It is an example of the hyperbolic localization.

\begin{Lemma}\label{lem:perverse}
  \textup{(1)} The functor $p_!i^*$ sends $P(\M_0)$ to $P(\M_0^{\C^*})$.

  \textup{(2)} Let $\bv'\le\bv$. The complex $p_!i^*
  \pi_!\cC_{\M(\bv',\bw)}$ has a canonical filtration whose associated
  graded is canonically identified with
  $\pi^{\C^*}_!\cC_{\M(\bv',\bw)^{\C^*}}$.
\end{Lemma}

This was proved in \cite[Lemma~5.1]{VV2} for quiver varieties of
finite type, but the proof actually gives the above statements for
general types.

Let us recall how the filtration is defined. Let us assume $\bv' =
\bv$ for brevity. Consider the diagram
\begin{equation*}
\begin{CD}
 \M @<i'<< \fT = \bigsqcup \fT_\alpha @>{\bigsqcup p_\alpha}>>
 \bigsqcup \M_\alpha = \M^{\C^*}
\\
  @V{\pi}VV  @V{\pi_{\fT}}VV @VV{\pi^{\C^*}}V
\\
 \M_0 @<<i< \fT_0 @>>p> \M_0^{\C^*}
\end{CD}
\end{equation*}
where $i'$ is the inclusion, $\pi_{\fT}$ is the restriction of $\pi$
to $\fT$, and $p_\alpha$ is the projection of the vector bundle
\eqref{eq:vector}.
Note that each $p_\alpha$ is a morphism, but the union $\bigsqcup
p_\alpha$ does {\it not\/} gives a morphism $\fT\to\M^{\C^*}$.

Recall the order $<$ on the set $\{\alpha\}$ of fixed point
components, and closed subvarieties $\fT_{\le\alpha}$, $\fT_{<\alpha}$
in \subsecref{subsec:variety}.
Let $\pi_{\le\alpha}$, $\pi_{<\alpha}$ be the restrictions of
$\pi_{\fT}$ to $\fT_{\le\alpha}$, $\fT_{<\alpha}$ respectively. Then
the main point in \cite[Lemma~5.1]{VV2} (based on \cite[\S4]{Lu-can2})
was to note that there is the canonical short exact sequence
\begin{equation}\label{eq:short}
  0 \to \pi^{\C^*}_! \cC_{\M_\alpha}
    \to (p\circ \pi_{\le\alpha})_! \cC_{\fT_{\le\alpha}}
    \to (p\circ \pi_{<\alpha})_! \cC_{\fT_{<\alpha}} 
  \to 0.
\end{equation}
Since $\fT_{\le\alpha} = \fT$ for the maximal element $\alpha$ and we
have $i^*\pi_!\cC_{\M} = \pi_{\fT!}i^{\prime *}\cC_{\M}$, this gives
the desired filtration.

During the proof it was also shown that $(p\circ \pi_{\le\alpha})_!
\cC_{\fT_{\le\alpha}}$, $(p\circ \pi_{<\alpha})_! \cC_{\fT_{<\alpha}}$
are semisimple.  (It is not stated explicitly in \cite{VV2}, but comes
from \cite[4.7]{Lu-can2}.)
Therefore the short exact sequence \eqref{eq:short} splits, and hence
$p_! i^* \pi_! \cC_{\M}$ and $\bigoplus_\alpha \pi^{\C^*}_!
\cC_{\M_\alpha} = \pi^{\C^*}_!\cC_{\M^{\C^*}}$ is isomorphic. The
choice of an isomorphism depends on the choice of splittings of the
above short exact sequences for all $\alpha$.

The exact sequence \eqref{eq:short} is the sheaf theoretic counterpart
of \eqref{eq:short1}. More precisely it is more natural to consider
the transpose of \eqref{eq:short1}:
\begin{equation}\label{eq:short3}
  0 \to (p\circ \pi_{<\alpha})_* \cC_{\fT_{<\alpha}} 
    \to (p\circ \pi_{\le\alpha})_* \cC_{\fT_{\le\alpha}}
    \to 
    \pi^{\C^*}_* \cC_{\M_\alpha}
  \to 0,
\end{equation}
obtained by applying the Verdier duality.

Recall that we study $H_{\topdeg}(Z_{\fT})$ in order to describe a
splitting of \eqref{eq:short1} by convolution.

\begin{Lemma}\label{lem:Yoneda}
  We have a natural isomorphism
  \begin{equation*}
    H_{\topdeg}(Z_{\fT}) \cong
    \Hom_{D(\M_0^{\C^*})}(p_! i^* \pi_! \cC_{\M},
    \pi^{\C^*}_{!} \cC_{\M^{\C^*}}).
  \end{equation*}
\end{Lemma}

The proof is exactly the same as \cite[Lemma~8.6.1]{CG}, once we use
the base change $i^*\pi_!\cC_{\M} = \pi_{\fT!}i^{\prime *}\cC_\M$.

This isomorphism is compatible with the convolution operator
\eqref{eq:op} in the following way:
Let $i_x$ denote the inclusion $\{x\}\to \M_0^{\C^*}$. Then an element
$c$ in 
\(
   \Hom_{D(\M_0^{\C^*})}(p_! i^* \pi_! \cC_{\M},
   \pi^{\C^*}_{!}\cC_{\M^{\C^*}})
   \cong
   \Hom_{D(\M_0^{\C^*})}(\pi^{\C^*}_{*}\cC_{\M^{\C^*}},
   p_* i^! \pi_* \cC_{\M})
\)
defines an operator
\begin{equation}\label{eq:op2}
  H^{p}(i_x^! \pi^{\C^*}_*\cC_{\M^{\C^*}}) \to
  H^{p}(i_x^!p_* i^! \pi_*\cC_{\M})
  \begin{NB}
  H^{-p}(i_x^*p_! i^* \pi_!\cC_{\M}) \to
  H^{-p}(i_x^* \pi^{\C^*}_!\cC_{\M^{\C^*}})
  \end{NB}
\end{equation}
by the Yoneda product. (See \cite[8.6.13]{CG}.)
We have
\begin{equation*}
  H^{p}(i_x^!p_* i^! \pi_*\cC_{\M}) \cong
  H^{p}(i_x^! (p\circ \pi_{\fT})_* i^{\prime !}\cC_\M)
  \cong
  H^{p}((p\circ \pi_{\fT})_* i_x^{\prime !} \cC_\M),
\end{equation*}
\begin{NB}
\begin{equation*}
  H^{-p}(i_x^*p_! i^* \pi_!\cC_{\M}) \cong
  H^{-p}(i_x^* (p\circ \pi_{\fT})_! i^{\prime *}\cC_\M)
  \cong
  H^{-p}((p\circ \pi_{\fT})_! i_x^{\prime*} \cC_\M),
\end{equation*}
\end{NB}%
where $i_x^{\prime}$ is the inclusion of $\fT_x$ in $\M$. The last one
is nothing but $H_{\topdeg-p}(\fT_x)$. Similarly
$H^{p}(i_x^! \pi^{\C^*}_*\cC_{\M^{\C^*}})$ is naturally isomorphic to
$H_{\topdeg-p}(\M^{\C^*}_x)$. Then we have
\begin{Lemma}
  Under the isomorphism in \lemref{lem:Yoneda}, the operator
  \eqref{eq:op2} given by $c\in H_{\topdeg}(Z_\fT)$ is equal to one in
  \eqref{eq:op}.
\end{Lemma}

The proof is the same as in \cite[\S8.6]{CG}.

The conditions \eqref{eq:cond} on $c\in H_{\topdeg}(Z_{\fT})$ is
translated into a language for the right hand side. We have the
following equivalent to the condition \eqref{eq:cond}:
\begin{subequations}\label{eq:cond2}
  \begin{align}
    & \text{$c$ maps
      $(p\circ\pi_{\le\alpha})_*\cC_{\fT_{\le\alpha}}$ to
      $\bigoplus_{\beta\le\alpha}\pi^{\C^*}_*\cC_{\M_\beta}$},
\\
      & \text{$c\colon (p\circ\pi_{\le\alpha})_*\cC_{\fT_{\le\alpha}}/
      (p\circ\pi_{<\alpha})_*\cC_{\fT_{<\alpha}}\to
      \pi^{\C^*}_*\cC_{\M_\alpha}$ is the identity.}
  \end{align}
\end{subequations}
Here the identity means the natural homomorphism given by
\eqref{eq:short3}.

Thus $c$ satisfying \eqref{eq:cond2} gives a splitting of
\eqref{eq:short} and hence an isomorphism
$p_!i^*\pi_!\cC_\M\cong \pi^{\C^*}_!\cC_{\M^{\C^*}}$. 
Therefore we have a graded algebra homomorphism
\begin{equation}\label{eq:Ext}
\begin{split}
  \Ext^\bullet_{D(\M_0)}(\pi_!\cC_\M,\pi_!\cC_\M)
  \xrightarrow{p_!i^*}
  &\Ext^\bullet_{D(\M_0^{\C^*})}(p_!i^*\pi_!\cC_\M,p_!i^* \pi_!\cC_\M)
\\
  &\xrightarrow[\cong]{\operatorname{Ad}(c)}
  \Ext^\bullet_{D(\M_0^{\C^*})}(\pi^{\C^*}_!\cC_{\M^{\C^*}},
  \pi^{\C^*}_!\cC_{\M^{\C^*}}).
\end{split}
\end{equation}
It is compatible with \eqref{eq:op2}, i.e.,
\begin{equation*}
  \begin{CD}
    H^{p}(i_x^! \pi^{\C^*}_*\cC_{\M^{\C^*}}) @>c>> 
    H^{p}(i_x^!p_* i^! \pi_*\cC_{\M})
\\
    @VaVV @VV{\operatorname{Ad}(c)p_*i^!(a)}V    
\\
    H^{p}(i_x^! \pi^{\C^*}_*\cC_{\M^{\C^*}}) @>c>> 
    H^{p}(i_x^!p_* i^! \pi_*\cC_{\M})
  \end{CD}
\end{equation*}
is commutative.

For $Z_{\fT^-}$ we have the following:

\begin{Lemma}\label{lem:T-}
We have natural isomorphisms
  \begin{equation*}
    \begin{split}
      H_{\topdeg}(Z_{\fT^-}) &\cong
      \Hom_{D(\M_0^{\C^*})}(\pi_!^{\C^*}\cC_{\M^{\C^*}},
      p^-_* i^{-!}\pi_!\cC_{\M})\\
      &\cong \Hom_{D(\M_0^{\C^*})}(\pi_!^{\C^*}\cC_{\M^{\C^*}},
      p_! i^{*}\pi_!\cC_{\M}).
    \end{split}
  \end{equation*}
\end{Lemma}

The first isomorphism is one as in \lemref{lem:Yoneda}. We exchange
the first and second factors, as we have changed the order of factors
$\M^{\C^*}$ and $\M$ containing $Z_{\fT^-}$. The sheaves are replaced
by their Verdier dual. The second isomorphism is induced by
\begin{equation*}
  p^-_* i^{-!}\pi_!\cC_{\M}\cong  p_! i^{*}\pi_!\cC_{\M},
\end{equation*}
proved by Braden \cite{Braden} (see Theorem 1 and the equation (1) at
the end of \S3).
\begin{NB}
  Applying $i_x^*$, we get
  \begin{equation*}
    H^{-p}(i_x^*p^-_* i^{-!}\pi_!\cC_{\M})\cong 
    H^{-p}(p_! i_x^{*}\pi_!\cC_{\M}) \cong H_{\topdeg-p}(\fT_x)^\vee.
  \end{equation*}
  As we cannot use the base change for the left hand side, we do not
  know another description for the left hand side.
\end{NB}%

We now have
\begin{Theorem}\label{thm:coproduct}
  The coproduct $\Delta_c$ in \eqref{eq:Delta} is equal to
  \eqref{eq:Ext}. In particular, $\Delta_c$ is an algebra
  homomorphism.
\end{Theorem}

\begin{proof}
  The isomorphisms in Lemmas~\ref{lem:Yoneda},\ref{lem:T-} are
  compatible with the product. Therefore, $c^{-1}\ast c =
  [\Delta_{\M^{\C^*}}]$ means that the composite
  \begin{equation*}
    \pi_!^{\C^*}\cC_{\M^{\C^*}}\xrightarrow{c^{-1}}
      p_! i^{*}\pi_!\cC_{\M}\xrightarrow{c}
    \pi_!^{\C^*}\cC_{\M^{\C^*}}
  \end{equation*}
  is the identity. (Note that the order of $c$, $c^{-1}$ is swapped as
  we need to consider the transpose of homomorphisms for convolution.)

  As $\pi_!^{\C^*}\cC_{\M^{\C^*}}$ and $p_! i^{*}\pi_!\cC_{\M}$ are
  semisimple, $c$, $c^{-1}$ can be considered as linear maps between
  isotypic components. (See \subsecref{subsec:desc} for explicit
  descriptions of isotypic components.) Therefore $c\circ c^{-1} =
  \id$ implies $c^{-1}\circ c = \id$ also. This, in particular, shows
  the uniqueness of $c^{-1}$ mentioned in
  \remref{rem:inverse}. Moreover this $c^{-1}$ is the inverse of $c$
  used in \eqref{eq:Ext}. Therefore $\Delta_c$ coincides with
  \eqref{eq:Ext} again thanks to the compatibility between the
  convolution and Yoneda products.
\end{proof}

\subsection{Coassociativity}

Since $\Delta_c$ depends on the choice of the class $c$, the
coassociativity does not hold in general. We give a sufficient
condition on $c$ (in fact, various $c$'s) to have the coassociativity
in this subsection.

Let $W = W^1\oplus W^2 \oplus W^3$ be a decomposition of the
$I$-graded vector space. Let $\bw = \bw^1+\bw^2+\bw^3$ be the
corresponding dimension vectors. Setting $W^{23} = W^2\oplus W^3$, we
have a flag $W^3\subset W^{23}\subset W$ with $W^3/W^{23} \cong W^2$,
$W/W^{23} = W^1$. This gives us a preferred order among factors
generalizing to $W^2\subset W$ in the previous setting.

The two dimensional torus $T = \C^*\times\C^*$ acts on $\M =
\M(\bv,\bw)$ through the homomorphism $\lambda\colon T\to G_W$ defined
by $\lambda(t_2,t_3) = \id_{W^1}\oplus t_2 \id_{W^2} \oplus t_3
\id_{W^3}$. 

We have two ways of putting braces for the sum $\bw =
(\bw^1+\bw^2)+\bw^3 = \bw^1 + (\bw^2 + \bw^3)$ respecting the order.
We have corresponding two $\C^*$'s in $T$ given by $\{ (1,t_3)\}$ and
$\{ (t_2,t_2) \}$. We denote the former by $\C^*_{12,3}$ and the
latter by $\C^*_{1,23}$. We then consider fixed points varieties,
tensor product varieties, and fiber products for both $\C^*$'s. We
denote them by $\M^{12,3}$, $\fT^{12,3}$, $Z_{\fT^{12,3}}$,
$\M^{1,23}$, $\fT^{1,23}$, $Z_{\fT^{1,23}}$, etc. They correspond to
block matrices
\(\left[
\begin{smallmatrix}
  * & * & *
\\
  * & * & *
\\
  0 & 0 & *  
\end{smallmatrix}\right]
\)
and
\(\left[
\begin{smallmatrix}
  * & * & *
\\
  0 & * & *
\\
  0 & * & *  
\end{smallmatrix}\right]
\)
respectively.

On these varieties, we have the action of the remaining $\C^* =
T/\C^*_{12,3}$ and $T/\C^*_{1,23}$ respectively. Then we can consider
the fixed point sets $(\M^{12,3})^{\C^*}$, $(\M^{12,3})^{\C^*}$. Both
are nothing but the torus fixed points $\M^{T}$. We denote it by
$\M^{1,2,3}$. We denote the corresponding fiber product
by $Z_{1,2,3}$.
In $\fT^{12,3}_0$, $\fT^{1,23}_0$, we consider subvarieties consisting
of points $\lim_{t\to 0}$ exists as before. They can be described as
the variety consisting of points $x=[B,a,b]$ such that $b_{\vin{h_N}}
B_{h_N}B_{h_{N-1}}\cdots B_{h_1}a_{\vout{h_1}}$ preserves the flag
$W^3\subset W^{23}\subset W$, i.e., \(\left[
\begin{smallmatrix}
  * & * & *
\\
  0 & * & *
\\
  0 & 0 & *  
\end{smallmatrix}\right].
\)
In particular, the variety is the same for one defined in
$\fT^{12,3}_0$ and in $\fT^{1,23}_0$. Therefore it is safe to write
both by $\fT^{1,2,3}_0$. We have the corresponding fiber product
\(
  Z_{\fT^{1,2,3}}\defeq \fT^{1,2,3}\times_{\M^{1,2,3}_0} \M^{1,2,3}.
\)

We need two more classes of varieties corresponding to
\(\left[
\begin{smallmatrix}
  * & * & 0
\\
  0 & * & 0
\\
  0 & 0 & *  
\end{smallmatrix}\right]
\)
and
\(\left[
\begin{smallmatrix}
  * & 0 & 0
\\
  0 & * & *
\\
  0 & 0 & *  
\end{smallmatrix}\right]
\)
respectively. Tensor product varieties are
\begin{equation*}
  \fT_0^{(1,2),3}\defeq \fT^{1,2,3}_0\cap \M^{12,3}_0,
\quad
  \fT_0^{1,(2,3)}\defeq \fT^{1,2,3}_0\cap \M^{1,23}_0
\end{equation*}
respectively.
We define the fiber products
$Z_{\fT^{(1,2),3}} = \fT^{(1,2),3}\times_{\M_0^{1,2,3}}\M^{1,2,3}$,
$Z_{\fT^{1,(2,3)}} = \fT^{1,(2,3)}\times_{\M_0^{1,2,3}}\M^{1,2,3}$.

A class $c^{12,3}\in H_{\topdeg}(Z_{\fT^{12,3}})$ gives the coproduct
\begin{equation*}
  \Delta_{c^{12,3}}\colon H_*(Z)\to H_*(Z_{12,3}),
\end{equation*}
and similarly $c^{1,23}\in
H_{\topdeg}(Z_{\fT^{12,3}})$ gives $\Delta_{c^{1,23}}$.
These correspond to $\Delta\otimes 1$ and $1\otimes \Delta$ for the
usual coproduct respectively.

A class $c^{(1,2),3}\in H_{\topdeg}(Z_{\fT^{(1,2),3}})$ gives
\begin{equation*}
  \Delta_{c^{(1,2),3}}\colon H_*(Z_{12,3})\to H_*(Z_{1,2,3}),
\end{equation*}
and similarly $c^{1,(2,3)}\in H_{\topdeg}(Z_{\fT^{(1,2),3}})$ gives
$\Delta_{c^{1,(2,3)}}$. Thus we have two ways going from $H_*(Z)$ to
$H_*(Z_{1,2,3})$:
\begin{equation}\label{eq:coass}
  \begin{CD}
    H_*(Z) @>\Delta_{c^{12,3}}>> H_*(Z_{12,3})
\\
    @V\Delta_{c^{1,23}}VV @VV\Delta_{c^{(1,2),3}}V
\\
   H_*(Z_{1,23}) @>>\Delta_{c^{1,(2,3)}}> H_*(Z_{1,2,3})   
  \end{CD}
\end{equation}
The commutativity of this diagram means the coassociativity of our
coproduct.

\begin{Proposition}\label{prop:coass}
  The diagram \eqref{eq:coass} is commutative if
  \begin{equation*}
    c^{12,3} \ast c^{(1,2),3} = c^{1,(2,3)}\ast c^{1,23}
  \end{equation*}
holds in $H_{\topdeg}(Z_{\fT^{1,2,3}})$.
\end{Proposition}

The proof is obvious.

\subsection{Equivariant homology version}\label{subsec:equivariant}

Let $G = \prod_i \GL(W^1_i)\times \GL(W^2_i)$. The group $G$ acts on
$\M$, $\M^{\C^*}$ and various other varieties considered in the
previous subsections.

We consider a $\C^*\times\C^*$-action on $\M$ defined by
\begin{equation*}
  (t_1,t_2)\cdot B_h =
  \begin{cases}
    t_1 B_h & \text{if $h\in \Omega$},\\
    t_2 B_h & \text{if $h\in \overline\Omega$},
  \end{cases}
\quad
    (t_1,t_2)\cdot a = a,
\quad
    (t_1,t_2)\cdot b = t_1 t_2 b.
\end{equation*}
Let $\mathbb G = \C^*\times\C^*\times G$.

\begin{Remark}
  When the graph does not contain a cycle, the action of a factor
  $\C^*$ of $\C^*\times\C^*$, lifted to the double cover, can be move
  to an action through $\C^*\to G$. Therefore we only have an action
  of $\C^*\times G$ essentially in this case.
  \begin{NB}
    See 2012-04-29 Note.
  \end{NB}
\end{Remark}

The results in the previous subsections hold in the equivariant
category: we replace the homology $H_*(X)$ by the equivariant homology
$H_*^{\mathbb G}(X)$. For the derived category $D(X)$ of complexes of
constructible sheaves, we use their equivariant version $D_{\mathbb
  G}(X)$, considered in \cite{BL,Lu-cus2}.
\begin{NB}
  When we consider the fibers $\fT_x$, $\M_x$, $\M^{\C^*}_x$, we need
  to replace $\mathbb G$ by a closed subgroup $A$ such that $x$ is
  fixed by $A$. But we do not need to consider fibers in the proof of
  \thmref{thm:coproduct}, which is most important.
\end{NB}%

The following observations are obvious, but useful. Top degree
components of $Z$ give a base for both $H_\topdeg(Z)$ and
$H_\topdeg^{\mathbb G}(Z)$. Therefore we have a natural isomorphism
\begin{equation*}
  H_\topdeg(Z_{\fT}) \cong H_\topdeg^{\mathbb G}(Z_{\fT}).
\end{equation*}
The corresponding statement for the right hand side of
\lemref{lem:Yoneda} is
\begin{equation*}
    \Hom_{D(\M_0^{\C^*})}(p_! i^* \pi_! \cC_{\M},
    \pi^{\C^*}_{!} \cC_{\M^{\C^*}})
    \cong
    \Hom_{D_{\mathbb G}(\M_0^{\C^*})}(p_! i^* \pi_! \cC_{\M},
    \pi^{\C^*}_{!} \cC_{\M^{\C^*}}).
\end{equation*}
This is also true as $p_! i^* \pi_! \cC_{\M}$, $\pi^{\C^*}_{!}
\cC_{\M^{\C^*}}$ are $\mathbb G$-equivariant perverse sheaves.
(See \cite[1.16(a)]{Lu-cus2}.)

In particular, $c\in H_{\topdeg}(Z_{\fT})$ defines the coproduct
$\Delta_c$ for the equivariant version $\Delta_c\colon H^{\mathbb
  G}_*(Z)\to H^{\mathbb G}_*(Z^{\C^*})$. Also to check the
coassociativity of the coproduct, we only need to check the condition
in \propref{prop:coass} for the {\it non\/}-equivariant homology.

\begin{Remark}
  In a wider framework of a holomorphic symplectic manifold with torus
  action satisfying certain conditions, Maulik and Okounkov \cite{MO}
  give a `canonical' element $c$. It is called the {\it stable
    envelop}. 
  It is defined first on the analog of $Z_{\fT}$ for the quiver
  varieties with generic complex parameters (deformations of $\M$,
  $\M^{\C^*}$), and then as the limit when parameters go to $0$.
  It satisfies \eqref{eq:cond} and the condition in
  \propref{prop:coass}. Therefore their stable envelop together with
  the construction in this section gives a canonical coproduct,
  satisfying the coassociativity.
\end{Remark}

\section{Tensor product multiplicities}

In this section, we give the formula of tensor product multiplicities
with respect to the coproduct $\Delta_c$ in terms of $IC$ sheaves.

\subsection{Decomposition of the direct image sheaf}

We give the decomposition of $\pi_!(\cC_\M)$ in this subsection. For
this purpose, we introduce a refinement of the stratification
\eqref{eq:decomp2}. We do not need to worry about the first factor
$\Mreg_0(\bv^0,\bw)$ as it cannot be decomposed further. On the other
hand the second factor $\M_0(\bv-\bv^0,0)$ parametrizes isomorphism
classes of semisimple modules $M$ of the preprojective algebra
corresponding to the quiver. They decompose into direct sum of simple
modules as
\begin{equation*}
  M = M_1^{\oplus n_1} \oplus M_2^{\oplus n_2} \oplus \cdots
  \oplus M_N^{\oplus n_N}.
\end{equation*}
Dimension vectors of all simple modules have been classified by
Crawley-Boevey \cite[Th.~1.2]{CB}. (In fact, he also classifies pairs
$(\bv^0,\bw)$ with $\Mreg_0(\bv^0,\bw)\neq\emptyset$.)
Let $\delta_1$, $\delta_2$, \dots, $\delta_N$ be such vectors which
are $\le\bv$.
They are all positive roots satisfying certain conditions. For
example, for a quiver of type $ADE$, they are simple roots. For a
quiver of affine type $ADE$, they are simple roots and the positive
generator $\delta$ of imaginary roots. For a Jordan quiver, it is
the vector $1\in\Z = \Z^I$.

We then have
\begin{equation*}
  \M_0(\bv-\bv^0,0) =
  S^{n_1}\Mreg_0(\delta_1,0)\times S^{n_2}\Mreg_0(\delta_2,0)\times\cdots
  \times S^{n_N}\Mreg_0(\delta_N,0),
\end{equation*}
with $\bv^0 + n_1\delta_1 + \cdots + n_N\delta_N = \bv$. Here
$\Mreg_0(\delta_k,0)$ parametrizes simple modules with dimension
vector $\delta_k$, or equivalently points in $\M_0(\delta_k,0)$ whose
stabilizers are nonzero scalars times the identity.
Its symmetric power $S^{n_k}\Mreg_0(\delta_k,0)$ parametrizes
semisimple modules
\begin{equation*}
  M_1^{\oplus m_1} \oplus M_2^{\oplus m_2} \oplus \cdots
\end{equation*}
such that $M_1$, $M_2$, \dots are distinct simple modules with
dimension $\delta_k$ and the total number of simple factors is $n_k$.

The symmetric power $S^{n_k}\Mreg_0(\delta_k,0)$ decomposes further
according to multiplicities $m_1$, $m_2$, \dots. As we may assume
$m_1\ge m_2 \ge\dots$, they define partition $\lambda_k$ of $n_k$. Let
us denote by $S_{\lambda_k}\Mreg_0(\delta_k,0)$ the space
parametrizing semisimple modules having multiplicities $\lambda_k$.

Thus we have
\begin{equation}\label{eq:strat}
  \M_0 = \bigsqcup
  \Mreg_0(\bv^0,\bw)\times \M_0(\vec{\lambda})
\end{equation}
with $\bv^0 + |\lambda_1|\delta_1 + \cdots + |\lambda_N|\delta_N = \bv$,
where
\begin{equation*}
  \M_0(\vec{\lambda}) \defeq
S_{\lambda_1} \Mreg_0(\delta_1,0)\times S_{\lambda_2}\Mreg_0(\delta_2,0)
   \times\cdots
   \times S_{\lambda_N}\Mreg_0(\delta_N,0).
\end{equation*}
This is nothing but the decomposition given in \cite[6.5]{Na-quiver},
\cite[3.27]{Na-alg}.

This stratification has a simple form when the quiver is of type
$ADE$. Each $\delta_k$ is a simple root $\alpha_i$, and
$\Mreg_0(\delta_k,0)$ is a one point given by the simple module
$S_i$. The symmetric product $S^{n_k}\Mreg_0(\delta_k,0)$ is also a
one point $S_i^{\oplus n_k}$, and hence we do not need to consider the
partition $\lambda_k$. Thus we can safely forget factors
$S_{\lambda_k}\Mreg_0(\delta_k,0)$ and get
\begin{equation*}
  \M_0 = \bigsqcup \Mreg_0(\bv^0,\bw),
\end{equation*}
with $\bv^0\le\bv$.

For the affine case $\delta_k$ is either simple root or $\delta$, as
we mentioned above. If $\delta_k$ is a simple root, we can forget the
factor $S^{n_k}\Mreg_0(\delta_k,0)$ as in the $ADE$ cases. If $\delta_k
= \delta$, then $\Mreg_0(\delta,0)$ is $\C^2$ for the Jordan quiver or
$\C^2\setminus\{0\}/\Gamma$ for the affine quiver corresponding to a
finite subgroup $\Gamma\subset\SU(2)$ via the McKay correspondence.
Therefore we have
\begin{equation}
  \M_0 = \bigsqcup \Mreg_0(\bv^0,\bw)
  \times \left(S_{\lambda}\C^2\text{ or }
  S_\lambda(\C^2\setminus\{0\})/\Gamma\right).
\end{equation}

Return back to a general quiver. We denote each stratum in
\eqref{eq:strat} by $\M_0(\bv^0;\vec{\lambda})$ for
brevity. Here $\vec{\lambda} = (\lambda_1,\dots,\lambda_N)$.
For a simple local system $\rho$ on this stratum, we consider
the corresponding $IC$ sheaf
\begin{equation*}
  IC(\M_0(\bv^0;\vec{\lambda}),\rho).
\end{equation*}
Then the decomposition theorem for a semismall projective morphism
\cite{BM} implies a canonical direct sum decomposition
\begin{equation}\label{eq:decompBM}
  \pi_!\cC_\M \cong
  \bigoplus IC(\M_0(\bv^0;\vec{\lambda}),\rho)
  \otimes H_{\topdeg}(\M_{x_{\bv^0;\vec{\lambda}}})_\rho.
\end{equation}
Here $x_{\bv^0;\vec{\lambda}}$ is a point in the stratum
$\M_0(\bv^0;\vec{\lambda})$ and $\M_{x_{\bv^0;\vec{\lambda}}} =
\pi^{-1}(x_{\bv^0;\vec{\lambda}})$ as before.  Then
$H_{\topdeg}(\M_{x_{\bv^0;\vec{\lambda}}})_\rho$ denotes the isotypic
component of $\rho$ in the homology group
$H_{\topdeg}(\M_{x_{\bv^0;\vec{\lambda}}})$ of the fiber with respect
to the monodromy action.

This decomposition determines representations of the convolution algebra
$H_{\topdeg}(Z) = \End_{D(\M_0)}(\pi_!\cC_\M)$ (see \cite[\S8.9]{CG}):
\begin{Theorem}\label{thm:semisimplealgebra}
  \textup{(1)} $\{
  H_{\topdeg}(\M_{x_{\bv^0;\vec{\lambda}}})_\rho\}$ is the set of
  isomorphism classes of simple modules of $H_{\topdeg}(Z)$.

  \textup{(2)} We have
  \begin{equation*}
    H_{\topdeg}(Z) \cong 
    \bigoplus \End(H_{\topdeg}(\M_{x_{\bv^0;\vec{\lambda}}})_\rho).
  \end{equation*}
\end{Theorem}

When the quiver is of type $ADE$, it was proved that only trivial
local systems on strata appear \cite[\S15]{Na-qaff} in the direct
summand of $\pi_!\cC_\M$, and hence we have
\begin{equation*}
  \pi_!(\cC_\M) \cong \bigoplus IC(\M_0(\bv^0,\bw))\otimes
  H_{\topdeg}(\M_{x_{\bv^0}}),
\end{equation*}
where we remove the local system $\rho$ from the notation for the $IC$
sheaves.

For a quiver of general type, the argument used in
\cite[\S15]{Na-qaff} implies that the simple local system $\rho$ is
trivial on the factor $\Mreg_0(\bv^0,\bw)$,
\begin{NB}
and hence comes only from
\(
  \M_0(\vec{\lambda}) = 
\)
i.e.,
\begin{equation*}
 IC(\M_0(\bv^0;\vec{\lambda}),\rho) 
 = IC(\M_0(\bv^0,\bw))\boxtimes
 IC(\M_0(\vec{\lambda}), \rho).
\end{equation*}
\end{NB}%
i.e., all simple modules $M_1$, $M_2$, \dots are of the form $S_i$.
In general, the author does not know what kind of local system $\rho$
can appear on these factors.
But we can show that only trivial local system appears for an affine
quiver:

\begin{Lemma}
  Suppose that the quiver is of affine type. Then
  \begin{equation*}
    \pi_!(\cC_\M) \cong \bigoplus_{\bv^0,\lambda}
    IC(\M_0(\bv^0,\bw))\boxtimes 
    \left(
      \cC_{\overline{S_{\lambda}(\C^2)}}
      \text{ or }
      \cC_{\overline{S_{\lambda}(\C^2/\Gamma)}}
      \right)
    \otimes H_{\topdeg}(\M_{x_{\bv^0;\lambda}}).
  \end{equation*}
\end{Lemma}

\begin{proof}
  By the argument in \cite[\S15]{Na-qaff}, it is enough to assume
  $\bv^0 = 0$ and hence $\Mreg_0(\bv^0,\bw)$ is a single point. Then a
  point in the stratum $x_{\bv^0;\lambda}$ is a point in
  $S_{\lambda}\C^2$ or $S_\lambda(\C^2\setminus\{0\})/\Gamma)$, and
  hence is written as $m_1 x_1 + m_2 x_2 + \cdots$, where $x_1$, $x_2$
  are distinct points in $\C^2$ or $\C^2\setminus\{0\})/\Gamma$. Then
  the fiber $\M_{x_{\bv^0;\lambda}}$ is the product of punctual
  Quot schemes parametrizing quotients $Q$ of the trivial rank $r$
  sheaf $\shfO_{\C^2}^{\oplus r}$ over $\C^2$ such that $Q$ is
  supported at $0$ and the length is $m_i$. Here $r$ is given by
  $\langle \bw, c\rangle$, where $c$ is the central element of the
  affine Lie algebra or $\bw$ itself for the Jordan quiver. This
  follows from the alternative description of quiver varieties of
  affine types, explained in \cite{MR1989196}. (Remark: In
  \cite[\S4]{MR1989196}, it was written that the fiber is the product
  of punctual Hilbert schemes, but it is wrong.) It is known that top
  degree part $H_{\topdeg}$ of a punctual Quot scheme is
  $1$-dimensional (see \cite[Ex.~5.15]{Lecture}). Therefore the
  monodromy action is trivial.
  Moreover $\overline{S_{\lambda}(\C^2)}$ and
  $\overline{S_{\lambda}(\C^2/\Gamma)}$ only have finite quotient
  singularities, and hence are rationally smooth. Therefore the
  intersection complexes are constant sheaves, shifted by dimensions.
\end{proof}

\subsection{A description of $H_\topdeg(Z_{\fT})$}\label{subsec:desc}

As in \thmref{thm:semisimplealgebra} we have a natural isomorphism
\begin{equation}\label{eq:decompTensor}
  H_\topdeg(Z_{\fT})
  \cong \bigoplus_{\bv^1,\bv^2,\vec{\lambda},\rho} \Hom(
  H_{\topdeg}(\M^{\C^*}_{x_{\bv^1,\bv^2;\vec{\lambda}}})_\rho,
  H_{\topdeg}(\fT_{x_{\bv^1,\bv^2;\vec{\lambda}}})_\rho)
\end{equation}
from \lemref{lem:Yoneda} and the above decomposition.

Thus $c\in H_\topdeg(Z_{\fT})$ is determined by its convolution action
$H_{\topdeg}(\M^{\C^*}_{x})\to H_{\topdeg}(\fT_{x})$ for $x =
x_{\bv^1,\bv^2;\vec{\lambda}}$ in each stratum. Then the converse of
\propref{prop:split} is clear.

\subsection{Tensor product multiplicities in terms of $IC$ sheaves}

As in the previous subsection, we also refine the stratification
in \lemref{lem:strat} as
\begin{equation*}
  \M^{\C^*}_0 = \bigsqcup 
  \Mreg_0(\bv^1,\bw^1)\times\Mreg_0(\bv^2,\bw^2)
  \times \M_0(\vec{\lambda}),
\end{equation*}
where
\begin{equation*}
  \M_0(\vec{\lambda}) = 
S_{\lambda_1} \Mreg_0(\delta_1,0)\times S_{\lambda_2}\Mreg_0(\delta_2,0)
   \times\cdots
   \times S_{\lambda_N}\Mreg_0(\delta_N,0)
\end{equation*}
as before.
For a simple local system $\rho$ on
$\Mreg_0(\bv^1,\bw^1)\times\Mreg_0(\bv^2,\bw^2)\times
\M_0(\vec{\lambda})$, we consider the corresponding $IC$ sheaf.  We
then have
 \begin{equation*}
    \begin{split}
  \pi^{\C^*}_!\cC_{\M^{\C^*}}
= 
  \bigoplus IC& (\Mreg_0(\bv^1,\bw^1)\times\Mreg_0(\bv^2,\bw^2)\times
  \M_0(\vec{\lambda}),\rho)
\\
  & \otimes
  H_{\topdeg}(\M^{\C^*}_{x_{\bv^1,\bv^2;\vec{\lambda}}})_\rho,
    \end{split}
  \end{equation*}
\begin{NB}
\begin{multline*}
  \pi^{\C^*}_!\cC_{\M^{\C^*}}
\\
= 
  \bigoplus IC(\M_0(\bv^1,\bw^1)\times\M_0(\bv^2,\bw^2))\boxtimes
  IC(\M_0(\vec{\lambda}),\rho) \otimes
  H_{\topdeg}(\M^{\C^*}_{x_{\bv^1,\bv^2;\vec{\lambda}}})_\rho,
\end{multline*}
\end{NB}%
where $x_{\bv^1,\bv^2;\vec{\lambda}}$ is a point in the stratum
$\Mreg_0(\bv^1,\bw^1)\times\Mreg_0(\bv^2,\bw^2)\times\M_0(\vec{\lambda})$.
Then
$H_{\topdeg}(\M^{\C^*}_{x_{\bv^1,\bv^2;\vec{\lambda}}})_\rho$
is a simple module of $H_{\topdeg}(Z^{\C^*})$, and any simple module
is isomorphic to a module of this form as before.

By $\Delta_c$ in \eqref{eq:Delta} we consider
$H_{\topdeg}(\M^{\C^*}_{x_{\bv^1,\bv^2;\vec{\lambda}}})_\rho$
as a module over $H_\topdeg(Z)$. Since $H_\topdeg(Z)$ is semisimple,
it decomposes into a direct sum of
$H_\topdeg(\M_{x_{\bv^0;\vec{\lambda}}})_{\rho'}$ with various
$\bv^0$, $\vec{\lambda'}$, $\rho'$. Let us define the `tensor product
multiplicity' by
\begin{equation}\label{eq:mult}
  n_{\bv^1,\bv^2;\vec{\lambda},\rho}^{\bv^0;\vec{\lambda'},\rho'}
  \defeq [H_{\topdeg}(\M^{\C^*}_{x_{\bv^1,\bv^2;\vec{\lambda}}})_\rho
  : H_\topdeg(\M_{x_{\bv^0;\vec{\lambda'}}})_{\rho'}].
\end{equation}

\begin{NB}
For a quiver of type $ADE$, we do not have $\vec{\lambda}$,
$\vec{\lambda'}$ and local systems $\rho$, $\rho'$, and we have
simpler definition:
\begin{equation*}
  n_{\bv^1,\bv^2}^{\bv^0}
  = [H_\topdeg(\pi^{-1}(x_{\bv^1,\bw^1}))\otimes
  H_\topdeg(\pi^{-1}(x_{\bv^2,\bw^2})):
   H_\topdeg(\pi^{-1}(x_{\bv^0}))],
\end{equation*}
where $x_{\bv^a,\bw^a}$ is a point in $\Mreg_0(\bv^a,\bw^a)$ ($a=1,2$).

This is not precise enough as $\pi^{-1}(x_{\bv^1,\bw^1})$ is taken in
$\M({}^\prime\bv^1,\bv^1)$ for various ${}^\prime\bv^1$.
\end{NB}

These multiplicity has a geometric description:
\begin{Theorem}\label{thm:tensor}
  The multiplicity
  $n_{\bv^1,\bv^2;\vec{\lambda},\rho}^{\bv^0;\vec{\lambda'},\rho'}$ is equal to
\begin{multline*}
  [p_!i^* IC(\M_0(\bv^0;\vec{\lambda'}),\rho'):
  IC(\Mreg_0(\bv^1,\bw^1)\times\Mreg_0(\bv^2,\bw^2)\times
  \M_0(\vec{\lambda}),\rho) ].
\end{multline*}
\end{Theorem}

Recall that $p_!i^* IC(\M_0(\bv^0;\vec{\lambda}),\rho)$ is a
direct sum of
\(
  IC(\Mreg_0(\bv^1,\bw^1)\times\Mreg_0(\bv^2,\bw^2)\times
  \M_0(\vec{\lambda'}),\rho')
\)
with various $\bv^1$, $\bv^2$, $\vec{\lambda'}$, $\rho'$ by
\lemref{lem:perverse}. The right hand side of the above formula
denote the decomposition multiplicity.

This formula is a direct consequence of decompositions of
$\pi_!\cC_\M$, $\pi^{\C^*}_!\cC_{\M^{\C^*}}$ and the identification of
$\Delta_c$ with $\operatorname{Ad}(c) p_!i^*$ in \eqref{eq:Ext}.
(See also \cite[Th.~5.1]{VV2}.)

\begin{Remark}
For a quiver of type $ADE$, we do not have data $\vec{\lambda}$,
$\rho$, $\vec{\lambda'}$, $\rho'$, and multiplicities
$n_{\bv^1,\bv^2}^{\bv^0}$ is nothing but the usual tensor product
multiplicity of finite dimensional representations of the Lie algebra
$\g$ of type $ADE$ \cite[Th.~5.1]{VV2}.

In general, the author does not know how to understand the behavior of
$IC(\M_0(\bv^0;\vec{\lambda}),\rho)$ under $p_!i^*$.
For affine types, only constant sheaves
$\cC_{\overline{S_\lambda(\C^2/\Gamma)}}$ appear in $\pi_!\cC_{\M}$,
and local
  systems on $\Mreg_0(\bv^1,\bw^1)\times\Mreg_0(\bv^2,\bw^2)\times
  S_\lambda(\C^2\setminus\{0\}/\Gamma)$ can be determined.
\begin{NB}
and general multiplicities
$n_{\bv^1,\bv^2;\vec{\lambda}}^{\bv^0;\vec{\lambda'}}$ can be
calculated from the special cases
$n_{\bv^1,\bv^2;\vec{\lambda}}^{\bv^0;\emptyset}$.
\end{NB}
It should be possible to determine multiplicities from the tensor product
multiplicity for the affine Lie algebra.
But it is yet to be clarified.
\end{Remark}

\subsection{Fixed point version}

Let $a$ be a semisimple element in the Lie algebra of $\mathbb
G$. Then it defines a homomorphism
\begin{equation*}
  \rho_a\colon H^*_{\mathbb G}(\mathrm{pt}) \to \C.
\end{equation*}
Let $A$ be the smallest torus whose Lie algebra contains $a$. Let
$Z^A$ be the fixed point set
Then we have a homomorphism
\begin{equation*}
  r_a\colon H^{\mathbb G}_*(Z)\otimes_{H^*_{\mathbb G}(\mathrm{pt})}\C
  \to H_*(Z^A)
\end{equation*}
as the composite of the pull back and the multiplication of $1\otimes
\rho_a(e(N))^{-1}$, where $N$ is the normal bundle of $\M^A$ in $\M$,
and $e(N)$ is its $A$-equivariant Euler class. (See
\cite[\S5.11]{CG}.) Then $r_a$ is an algebra isomorphism. Similarly we have
\begin{equation*}
  r_a\colon H^{\mathbb G}_*(Z^{\C^*})\otimes_{H^*_{\mathbb G}(\mathrm{pt})}\C
  \to H_*((Z^{\C^*})^A).
\end{equation*}

We then have a specialized coproduct
\begin{equation*}
  \Delta_c\colon H_*(Z^A) \to H_*((Z^{\C^*})^A).
\end{equation*}

Those convolution algebras can be studied in terms of perverse sheaves
appearing whose shifts appear in direct summand in $\pi^A\C_{\M^A}$,
$(\pi^{\C^*})^A_!\C_{(\M^{\C^*})^A}$, where $\pi^A$, $(\pi^{\C^*})^A$
are restrictions of $\pi$ and $\pi^{\C^*}$ to $A$-fixed point sets
$\M^A$ and $(\M^{\C^*})^A$. See \cite[\S8.6]{CG} for detail.

The tensor product multiplicities with respect to the specialized
$\Delta_c$ are described by the functor $p^A_! (i^A)^*$, where $p^A$,
$i^A$ are restrictions of $p$ and $i$ to $A$-fixed point sets. Since
the result is almost the same as \thmref{thm:tensor}, we omit the
detail. The difference is that the algebra is not semisimple in
general, and multiplicities are considered in the Grothendieck group
of the category of modules of convolution algebras. In geometric side,
perverse sheaves are {\it not\/} preserved by the functor $p^A_!
(i^A)^*$. They are sent to direct sums of {\it shifts\/} of perverse
sheaves in general.

As we mentioned in the introduction, the target of $\Delta_c$ in
\eqref{eq:Delta} is $H_*(Z^{\C^*})$, which is larger than the tensor
product of the corresponding algebra for $\bw^1$, $\bw^2$ in
general. This is because of the existence of the third factor in
\lemref{lem:strat}(1). To avoid this, we assume that generators
$\tr(B_{h_N}B_{h_{N-1}}\cdots B_{h_1}\colon V_{\vout{h_1}} \to
V_{\vin{h_N}} = V_{\vout{h_1}})$ have nontrivial weights with respect
to $A$. Then the $A$-fixed point set in the third factor
$\M_0(\bv-\bv^0,0)$ is automatically trivial, and hence we have
\begin{equation*}
  (Z^{\C^*})^A = \bigsqcup_{\bv^1+\bv^2=\bv} Z(\bv^1,\bw^1)^A\times
  Z(\bv^2,\bw^2)^A.
\end{equation*}
This assumption is rather mild and satisfied for example if the
compositions of $A\to \mathbb G$ with the projections $\mathbb G\to
\C^*$ to the first and second factor of $\mathbb G$ both have positive
weights. This condition occurs when we study modules of $Y(\g)$ for
example, as both are identities in that case.

\begin{NB}
\section{Coproduct on Yangian}

In this section we give a natural element $c$ satisfying the
conditions \eqref{eq:cond} in a slightly smaller subspace of
$H_{\topdeg}(Z_{\fT})$ under the assumption that the graph has no edge
loops.
It is related to a coproduct on two paremeter version of the Yangian.
The argument is based on \cite[\S6]{Na-Tensor}. We assume that the
reader is familiar with it.

Let $\g$ be a symmetrizable Kac-Moody Lie algebra associated with the
Cartan matrix $(a_{ij})$.
We assume $\g$ is simply-laced, i.e., $a_{ij} = a_{ji}$ and $a_{ij} =
0$ or $-1$ for $i\neq j$. We can probably modify the argument below so
that it works only under the symmetric assumption.

\subsection{Definition of the Yangian associated with a Kac-Moody Lie
  algebra}

We fix an invariant inner product $(\ ,\ )$ on $\g$.
We normalize Chevalley generators $x_i^\pm$, $h_i$ so that
$(x_i^+,x_i^-) = 1$ and $h_i = [x^+_i,x^-_i]$. This is different from the
usual convention (as in \cite{Kac}), but is convenient for the
definition of the Yangian $Y(\g)$ as the assignment $x^\pm_i,
h_i \mapsto x^\pm_{i 0}, h_{i,0}$ gives an algebra homomorphism
$U(\g)\to Y(\g)$.
\begin{NB2}
  Let $e_i$, $f_i$, $\alpha_i^\vee$ be generators in \cite{Kac}. Let
  $\epsilon_i = 2/(\alpha_i,\alpha_i)$ as in \cite[(2.1.6)]{Kac}. Then
  we take $x_i^+ = \sqrt{(\alpha_i,\alpha_i)/2}\, e_i =
  \epsilon_i^{-1/2} e_i$, $x_i^- = \sqrt{(\alpha_i,\alpha_i)/2}\, f_i =
  \epsilon_i^{-1/2} f_i$, $h_i = (\alpha_i,\alpha_i)/2\, \alpha_i^\vee =
  \epsilon_i^{-1} \alpha_i^\vee$. Then we have $(x^+_i, x^-_i) =
  \epsilon_i^{-1} (e_i,f_i) = 1$.
\end{NB2}

We also choose an orientation $\Omega$ of the Dynkin diagram
corresponding to $(a_{ij})$.

Let $\g'$ be the derived subalgebra $[\g,\g]$.
We first define the Yangian $Y(\g')$ associated with $\g'$.
It is the associative algebra over $\C[\ve_1,\ve_2]$ with generators
$x_{i r}^\pm$, $h_{i r}$ ($i\in I$, $r\in\Z_{\ge 0}$) with the
following defining relations {\allowdisplaybreaks[4]
\begin{gather}
  [h_{i r}, h_{j s}] = 0, \label{eq:relHH}\\
  [h_{i 0}, x^\pm_{j s} ] = \pm (\alpha_i,\alpha_j) x^\pm_{j s},
  \label{eq:relHX}
\\
  [x^+_{i r}, x_{j s}^-] = \delta_{ij} h_{i, r+s}, \label{eq:relXX}
\\
  [h_{i,r+1}, x^\pm_{js}] - [h_{i,r}, x^\pm_{j,s+1}] 
  = 
  \mp \ve_{i\to j} x^\pm_{js} h_{ir}
  \mp \ve_{j\to i} h_{ir} x^\pm_{js}, \label{eq:relexHX}
\\
  [x^\pm_{i, r+1}, x^\pm_{j s}] - [x^\pm_{i r}, x^\pm_{j, s+1}] = 
  \mp \ve_{i\to j} x^\pm_{js} x_{ir}
  \mp \ve_{j\to i} x_{ir} x^\pm_{js},
   \label{eq:relexXX}
\\
  \sum_{\sigma\in S_b}
   [x^\pm_{i r_{\sigma(1)}}, [ x^\pm_{i r_{\sigma(2)}}, \cdots,
       [x^\pm_{i, r_{\sigma(b)}}, x^\pm_{j s}] \cdots ]]
   = 0
   \quad \text{if $i\neq j$,}
 \label{eq:relDS}
\end{gather}
where} $b = 1-a_{ij}$, $\ve_{i\to j} = \ve_1$ if $i\to j\in\Omega$,
$\ve_2$ if $i\to j\in\overline\Omega$, and $0$ if $a_{ij} = 0$.
\begin{NB2}
  I need to add elements in $\mathfrak h$ when the Cartan matrix is
  not invertible.
\end{NB2}

We add the Cartan subalgebra $\mathfrak h$ of $\mathfrak g$ to
$Y(\g')$ to define $Y(\g)$ with the relations
\begin{equation}\label{eq:Cartan}
\begin{gathered}
  h_{i 0} = \frac{(\alpha_i,\alpha_i)}2 \alpha_i^\vee,
\\
  [h,h_{i r}] = 0, \quad
  [h, x^\pm_{i r}] = \langle \alpha_i, h\rangle x^\pm_{i r} \quad
  \text{for $h\in\mathfrak h$}.
\end{gathered}
\end{equation}

The coproduct $\Delta$ is defined only when $\g$ is finite dimensional or
affine so far as mentioned in Introduction.
But consider the specialization $Y_0(\g)$ of $Y(\g)$ at $\ve_1 = \ve_2
= 0$. Then $Y_0(\g)$ is the universal enveloping algebra of the Lie
algebra defined by the same relations as above except the right hand
sides of \eqref{eq:relexHX} and \eqref{eq:relexXX} are replaced by
$0$. It has a coproduct $\Delta$ defined by
\begin{gather*}
  \Delta(h_{ir}) = h_{ir}\otimes 1 + 1\otimes h_{ir}, \quad
  \Delta(x^\pm_{ir}) = x^\pm_{ir}\otimes 1 + 1\otimes x^\pm_{ir}, \\
  \Delta(h) = h\otimes 1+ 1\otimes h \quad (h\in\mathfrak h).
\end{gather*}
It is straightforward to check that this is compatible with the
defining relations.

\subsection{Review of the argument in \cite[\S6]{Na-Tensor}}

We fix $W$ and $W = W^1\oplus W^2$ as in \secref{sec:}, but do not fix
$\bv$. We consider $\M(\bv,\bw)$ with various $\bv$ simultaneously:
\begin{equation*}
  \M(\bw) \defeq \bigsqcup_{\bv} \M(\bv,\bw), \quad
  \M_0(\infty,\bw) \defeq \bigcup_{\bv} \M_0(\bv,\bw),
\end{equation*}
where we have used the closed embedding
$\M_0(\bv,\bw)\subset\M_0(\bv',\bw)$ for $\bv\le\bv'$ given by the
extension by $0$ for the second definition. We also consider
\begin{multline*}
  \fT(\bw) \defeq \bigsqcup_\bv \fT(\bv,\bw),
\qquad
  Z(\bw) \defeq \M(\bw)\times_{\M_0(\infty,\bw)}\M(\bw).
\end{multline*}

Let $\mathbb G$ as in \subsecref{subsec:equivariant}.
Varagnolo \cite{Varagnolo} constructed an algebra
homomorphism
\begin{equation*}
  Y(\g)\to H_*^{\mathbb G}(Z(\bw)),
\end{equation*}
based on an earlier work on the quantum loop algebra by the author
\cite{Na-qaff}. The parameters $\ve_1$, $\ve_2$ are identified with
variables for the $\C^*\times\C^*$-equivariant cohomology group of a
point.
(See \cite[\S5]{Na-Tensor} for the precise definition of $H_*^{\mathbb
  G}(Z(\bw))$.)
Here the simply-laced assumption is used. If we merely assume $\g$ is
symmetric, the defining relations of $Y(\g)$ must be
modified.

We specialize the homomorphism at $\ve_1 = \ve_2 = 0$:
\begin{equation*}
  Y_0(\g) \to H^G_*(Z(\bw)).
\end{equation*}
In particular, we have $Y_0(\g)$-module structures on 
$H^G_*(\La(\bw))$ and $H^G_*(\widetilde\fT(\bw))$. They have
distinguished vectors given by the fundamental class of
$\La(0,\bw) = \fT(0,\bw) = \mathrm{pt}$. We denote it by $[0]_\bw$.

On the other hand, we have a $Y_0(\g)\otimes Y_0(\g)$-module structure
on $H^G_*(\La(\bw^1))\otimes_{H_G^*(\mathrm{pt})} H^G_*(\La(\bw^2))$
with a distinguished vector $[0]_{\bw^1}\otimes[0]_{\bw^2}$. Therefore
we have an induced $Y_0(\g)$-module structure via the coproduct
$\Delta$ in the previous subsection.

\begin{Lemma}
  Let $\mathscr S(G)$ be the quotient field of
  $H^*_G(\mathrm{pt})$. We have a unique $Y_0(\g)\otimes\mathscr
  S(G)$-module isomorphism
  \begin{equation*}
   \Phi_{\mathscr S}\colon
    H^G_*(\La(\bw^1))\otimes_{H_G^*(\mathrm{pt})} H^G_*(\La(\bw^2))
    \otimes\mathscr S(G)
    \to H^G_*(\widetilde\fT(\bw))\otimes\mathscr S(G)
  \end{equation*}
sending the class $[0]_{\bw^1}\otimes [0]_{\bw^2}\to [0]_\bw$.
\end{Lemma}

This is proved as in \cite[6.4]{Na-Tensor}.

Let $\widetilde\fT_1(\bv,\bw) = \widetilde\fT(0,\bw^1;\bv,\bw^2)$. It
is a {\it closed\/} stratum of $\widetilde\fT(\bv,\bw)$. It is the
total space of a vector bundle over $\La(0,\bw^1)\times\La(\bv,\bw^2) \cong
\La(\bv,\bw^2)$. Hence we have the Thom isomorphism
\begin{equation*}
  H^G_*(\widetilde\fT_1(\bv,\bw))\cong H^G_*(\La(\bv,\bw^2))
\end{equation*}
and a homomorphism
\begin{equation*}
  H^G_*(\widetilde\fT_1(\bv,\bw))\to H^G_*(\widetilde\fT(\bv,\bw))
\end{equation*}
given by the inclusion.

Let $\fT_1(\bw) = \bigsqcup \fT_1(\bv,\bw)$.

\begin{Lemma}
  The following diagram is commutative:
  \begin{equation*}
    \begin{CD}
      H^G_*(\La(\bw^2)) @>\text{Thom isom.}>> H^G_*(\widetilde\fT_1(\bw))
\\
    @V{[0]_{\bw^1}\otimes\bullet}VV @VVV
\\
      H^G_*(\La(\bw^1))\otimes_{H_G^*(\mathrm{pt})} H^G_*(\La(\bw^2))
      @. 
      H^G_*(\widetilde\fT(\bw))
\\
    @V{\otimes{\mathscr S(G)}}VV @VV{\otimes{\mathscr S(G)}}V
\\
   H^G_*(\La(\bw^1))\otimes_{H_G^*(\mathrm{pt})} H^G_*(\La(\bw^2))
    \otimes\mathscr S(G)
    @>>\Phi_{\mathscr S}>
   H^G_*(\widetilde\fT(\bw))\otimes\mathscr S(G).
    \end{CD}
  \end{equation*}
\end{Lemma}

The proof is the same as that of \cite[6.9]{Na-Tensor} except one
modification. The fundamental class of $\La(\bw^2-\mathbf m^2,\bw^2)$
giving the lowest weight vector was used in the proof. This class
makes sense only when $\g$ is finite dimensional. We use instead
fundamental classes of $\La(\bv^2,\bw^2)$ giving extremal weight
vectors of weights $w\cdot \bw$ for any Weyl group element $w$.
Then the argument of \cite[6.9]{Na-Tensor} shows that the diagram is
commutative on the subspace $Y_0(\g)^+ [\La(\bv^2,\bw^2)]$ where
$Y_0(\g)^+$ is the subalgebra generated by $x^+_{ir}$ ($i\in I$, $r\ge
0$).
This subspace is an analog of Demazure modules for a Kac-Moody Lie
algebra.
It is easy to show that the union of subspaces for all $w$
cover the whole $H_*(\La(\bw^2))$.
\begin{NB2}
  By the commutation relations, we have $x^-_{i,r}Y_0(\g)^+
  [\La(\bv^2,\bw^2)] \subset Y_-(\g)^+ x^-_{i,r}[\La(\bv^2,\bw^2)] +
  Y_-(\g)^+ [\La(\bv^2,\bw^2)]$. Then if
  $x^-_{i,r}[\La(\bv^2,\bw^2)]\neq 0$, we take the `next' extremal
  vector $(x^-_{i 0})^m [\La(\bv^2,\bw^2)]$ for some $m$. Then we see
  \begin{equation*}
    \begin{split}
   & x^-_{i,r}[\La(\bv^2,\bw^2)]
    = x^-_{i,r} (x^+_{i 0})^{m} (x^-_{i 0})^{m} [\La(\bv^2,\bw^2)]
\\
   & = (x^+_{i 0} x^-_{i r} - h_{i r})(x^+_{i 0})^{m-1}
    (x^-_{i 0})^{m} [\La(\bv^2,\bw^2)].
    \end{split}
  \end{equation*}
  By induction we see that this is in $Y_0(\g)^+(x^-_{i 0})^{m}
  [\La(\bv^2,\bw^2)])$.
\end{NB2}%
Therefore we get the assertion.

Now as in \cite[6.12]{Na-Tensor} the middle row in the above
commutative diagram can be completed. We get
\begin{NB2}
$H^G_*(\La(\bw^2))$ and $H^G_*(\widetilde\fT_1(\bw))$ span
$H^G_*(\La(\bw^1))\otimes_{H_G^*(\mathrm{pt})} H^G_*(\La(\bw^2))$ and
$H^G_*(\widetilde\fT(\bw))$ as $Y_0(\g)$-modules respectively. 
\end{NB2}%

\begin{Theorem}\label{thm:tensor}
  $\Phi_{\mathscr S}$ induces an isomorphism
  \begin{equation*}
    \Phi\colon 
    H^G_*(\La(\bw^1))\otimes_{H_G^*(\mathrm{pt})} H^G_*(\La(\bw^2))
      \to
    H^G_*(\widetilde\fT(\bw))
  \end{equation*}
of $Y_0(\g)\otimes H_G^*(\mathrm{pt})$-modules.
\end{Theorem}

Forgetting the $G$-action, we still have an isomorphism between
non-equivariant homology groups. It is degree preserving, and hence we
get an isomorphism
\begin{equation}\label{eq:Phi}
    H_{\topdeg}(\La(\bw^1))\otimes H_{\topdeg}(\La(\bw^2))
  \xrightarrow[\Phi]{\cong} H_{\topdeg}(\widetilde\fT(\bw)).
\end{equation}
This is an isomorphism of $\g$-modules.

Recall the decomposition \eqref{eq:decomp}. We define
\begin{equation*}
  \widetilde\fT_{\le\bv^1} \defeq
  \bigsqcup_{\bv'\le\bv^1,\bv^2} \fT(\bv',\bw^1;\bv^2,\bw^2).
\end{equation*}

\begin{Lemma}
  Fix $\bv^1$. Then $\Phi$ induces an isomorphism
  \begin{equation*}
    \bigoplus_{\bv'\le \bv^1} H_\topdeg(\La(\bv',\bw^1))\otimes
    H_\topdeg(\La(\bw^2)) \xrightarrow{\cong}
    H_{\topdeg}(\widetilde\fT_{\le \bv^1})
  \end{equation*}
\end{Lemma}

\begin{proof}
If $\bv^1 = 0$, this part of $\Phi$ is given by the Thom isomorphism
\(
  H_\topdeg(\La(\bw^2)) \cong H_\topdeg(\widetilde\fT_1).
\)
This is nothing but the above assertion.

For general $\bv^1$, the left hand side is spanned by
  \begin{equation*}
    f_{i_1}\cdots f_{i_k}\left(
    H_\topdeg(\La(0,\bw^1))\otimes_{H^*_G(\mathrm{pt})}
    H_\topdeg(\La(\bw^2))\right)
  \end{equation*}
  with various $i_1$, \dots, $i_k$ such that $\alpha_{i_1} + \dots +
  \alpha_{i_k} = \bv^1$. From the definition of the operators
  $f_{i_1}$, \dots, $f_{i_k}$, $f_{i_1}\cdots
  f_{i_k}H_\topdeg(\widetilde\fT_1)$ is supported by the subvariety
  consisting of $[B, a, b]\in\M(\bw)$ such that there exists a
  filtration $V\supset {}^kV\supset\cdots \supset {}^1V$ of $I$-graded
  subspaces with $\dim V/{}^kV = \alpha_{i_k}$, \dots, $\dim
  {}^2V/{}^1V = \alpha_{i_1}$, which is $B$-invariant, satisfies $\Ima
  a\subset V^1$ and the restriction of $[B,a,b]$ to $V^1$ is contained
  in $\widetilde\fT_1$.
\end{proof}

\subsection{$c$ from the coproduct}

On the other hand, we consider the isomorphism $\Phi$ in
\thmref{thm:tensor} forgetting the $G$-action. From our definition of
$\Phi$, it preserves the degree. So we have
This gives an element in the summand for $x_{\bv^1,\bv^2;\lambda} = 0$
(with trivial $\rho$) in \eqref{eq:decompTensor}.

Consider more generally a point $x_{\bv^1,\bv^2} =
(x_{\bv^1},x_{\bv^2})\in\Mreg_0(\bv^1,\bw^1)\times\Mreg_0(\bv^2,\bw^2)$. Modifying
the argument in the previous subsection, or using the slice argument
in \cite[\S3.3]{Na-qaff} to show that fibers over $x_{\bv^1,\bv^2}$
for $\M^{\C^*}$ and $\fT$ are isomorphic to central fibers
$\La(\bw^{\prime1})\times\La(\bw^{\prime2})$, $\widetilde\fT(\bw')$
for suitable $\bw^{\prime1}$, $\bw^{\prime2}$, $\bw'$, we have an
analogous isomorphism
\begin{equation*}
  H_{\topdeg}(\M(\bw^1)_{x_{\bv^1}})\otimes H_{\topdeg}(\M(\bw^2)_{x_{\bv^2}})
  \xrightarrow{\cong} H_{\topdeg}(\fT(\bw)_{x_{\bv^1,\bv^2}}).
\end{equation*}
Thus we have elements in the summand in \eqref{eq:decompTensor} for
$x_{\bv^1,\bv^2;\lambda}$ with $\lambda=\emptyset$ (and trivial $\rho$).

From the definition of $\Phi$,

Motivated by the above construction, we introduce a subvariety
$Z^{\mathrm{reg}}_\fT\subset Z_\fT$ defined as
\begin{equation*}
   \bigcup_{\bv^1,\bv^2} 
   \left\{ (x^1,x^2)\in Z_{\fT} \middle| \pi(x^1) = \pi(x^2)\in
   \Mreg_0(\bv^1,\bw^1)\times\Mreg_0(\bv^2,\bw^2)  \right\}.
\end{equation*}
Then the above construction gives an element

\begin{NB2}
\begin{equation*}
   \bigcup_{\vec{\lambda}\neq 0} 
   \left\{ (x^1,x^2)\in Z_{\fT} \middle| \pi(x^1) = \pi(x^2)\in
   \M_0(\bv^0;\vec{\lambda})  \right\}.
\end{equation*}
\end{NB2}
\end{NB}

\subsection*{Acknowledgments}

The author thanks D.~Maulik and A.~Okounkov for discussion on their
works.

\bibliographystyle{myamsplain}
\bibliography{nakajima,mybib,tensor2}

\end{document}